\documentclass[11pt,a4paper]{article}

\usepackage[title]{appendix}
\usepackage[export]{adjustbox}
\usepackage{amssymb}
\usepackage{amsmath}
\usepackage{amsfonts,amsthm}
\usepackage[all]{xy}
\usepackage{float}
\usepackage{graphicx}
\usepackage{environ}
\usepackage{color}
\usepackage{tikz}
\usepackage{tikzscale}
\usepackage{pgfplots,enumitem,color,colortbl}
\usepackage{setspace}
\usepackage{authblk}

\usepackage{booktabs,tabularx,tabulary,multirow}

\pgfplotsset{compat=1.14}
\usepackage[margin=1in]{geometry}
\usetikzlibrary{automata,arrows,calc,positioning}
\input{epsf.tex}
\input xy
\xyoption{all}

\let\OLDthebibliography\thebibliography
\renewcommand\thebibliography[1]{
  \OLDthebibliography{#1}
  \setlength{\parskip}{0pt}
  \setlength{\itemsep}{0pt plus 0.3ex}
}
\setitemize{noitemsep}

\newcommand{\specialcell}[2][l]{%
  \begin{tabular}[#1]{@{}l@{}}#2\end{tabular}}

\newtheorem{thm} {\noindent \bf{Theorem}}[section]

\newtheorem{prop}{\noindent \bf{Proposition}}[section]

\renewcommand{\(}{\left(}
\renewcommand{\)}{\right)}

\newcommand{\up}{$\uparrow$ }
\newcommand{\down}{$\downarrow$ }
\newcommand{\uni}{$\wedge$ }
\newcommand{\x}{$\times$ }
\renewcommand{\b}{\bfseries}
\definecolor{LightCyan}{rgb}{0.8,1,1}
\renewcommand{\c}{\cellcolor{LightCyan}}
\definecolor{LightCyan2}{rgb}{0.6,1,1}
\renewcommand{\d}{\cellcolor{LightCyan2}}
\numberwithin{equation}{section} 
\makeatletter
\newsavebox{\measure@tikzpicture}
\NewEnviron{scaletikzpicturetowidth}[1]{%
  \def\tikz@width{#1}%
  \def\tikzscale{1}\begin{lrbox}{\measure@tikzpicture}%
  \BODY
  \end{lrbox}%
  \pgfmathparse{#1/\wd\measure@tikzpicture}%
  \edef\tikzscale{\pgfmathresult}%
  \BODY
 }
\makeatother

\numberwithin{equation}{section}

\begin{document}
\title{Strategic customer behavior in a queueing system with alternating information structure}
\author[1]{Yiannis Dimitrakopoulos}
\author[1]{Antonis Economou}
\author[2]{Stefanos Leonardos}
\affil[1]{National and Kapodistrian University of Athens, 
Department of Mathematics\\
Panepistemiopolis, Athens 15784, Greece}

\affil[2]{Singapore University of Technology and Design,
8 Somapah Rd, 487372 Singapore}

\affil[ ]{\{dimgiannhs, aeconom\}@math.uoa.gr; stefanos\_leonardos@sutd.edu.sg}

\date{\today}

\maketitle

\noindent \textbf{Abstract:} Strategic customer behavior is strongly influenced by the level of information that is provided to customers. Hence, to optimize the design of queueing systems, many studies consider various versions of the same service model and compare them under different information structures. In particular, two extreme versions are usually considered and compared: the observable in which customers are informed about the number of customers in the system and the unobservable in which they are only informed about the system parameters, e.g., arrival and service rates. In the present work, we study a model that bridges these two versions. More concretely, we assume that the system alternates between observable and unobservable periods. We characterize and compute customer equilibrium joining/balking strategies and show that the present model unifies and extends existing approaches of both heterogeneously observable models and models with delayed observations. More importanly, our findings indicate that an alternating information structure implies in general higher equilibrium throughput and social welfare in comparison to both the observable and unobservable cases. We complement our results with numerical experiments and provide managerial insight on the optimal control of the system parameters.

\vspace{0.5cm}


\section{Introduction}

The strategic customer behavior in queueing systems has received considerable attention since the pioneering paper of Naor \cite{Na1969}, who studied the M/M/1 queue from an economic viewpoint. Naor assumed that the customers are active entities who decide whether to join or balk after observing the queue length. He also considered the problem of a social planner and a monopolist who take into account the customer strategic behavior when they aim to optimize the social welfare and the profit respectively. The study of this observable version of the M/M/1 queue was subsequently complemented by Edelson and Hildebrand \cite{EdHi1975} who considered the same problems for the unobservable version of the system. In their model, the arriving customers are not allowed to observe the number of customers in the system and make their decisions relying solely on the knowledge of its operational and economic parameters. Since then, the literature has grown considerably. Hassin and Haviv \cite{HaHa2003} and Hassin \cite{Ha2016} survey the basic methodology and results till 2003 and from 2003 till 2016 respectively. Stidham \cite{St2009}, in his monograph on optimal design of queueing systems, provides a comprehensive overview of the various fundamentals models in the area.\par
 
One recurrent theme in the strategic queueing literature is the study of the effect of the level of information that is provided to the customers on their strategic behavior. This is an important theoretical issue per se, but it is also important for a social planner or a monopolist that are interested in the optimal design of a given system. What is interesting is that the effect of the level of information is ambiguous. Using the framework of the M/M/1 queue to reduce the complexity of the problem, existing studies demonstrate that more information can benefit or hurt the customers and/or the service provider, depending on various parameters and structural assumptions of the underlying model, see e.g., Hassin \cite{Ha1986}, Chen and Frank \cite{ChFr2004}, Guo and Zipkin \cite{GuZi2007}.\par
The results in this literature show that neither the observable nor the unobservable versions of the M/M/1 queue are preferable for the whole range of the underlying operational and economic parameters. To gain further insight, a number of authors studied the M/M/1 queue with strategic customers under information structures that lie between the observable and unobservable versions. To the best of our knowledge, there are three main ideas that have appeared in the literature that bridge the observable and unobservable versions of the M/M/1 queue: partially observable models \cite{EcKa2008,GuZi2009,SiHaStZh2016,KiKi2017,HaKo2014}, heterogeneously observable models \cite{EcGr2015,HuLiWa2018} and models with delayed observations \cite{BuEcVa2017,HaRo2014,Ro2013}. We provide a brief overview of these results in Subsection \ref{sub:review}. \par
In the present paper, we develop an alternative model that bridges observable and unobservable models and unifies preexisting approaches. Specifically, we consider a queueing system of M/M/1 type with strategic customers, which alternates between observable and unobservable periods: customers that arrive during observable periods see the number of present customers before making their joining/balking decisions, whereas customers that arrive during unobservable periods are only informed about the system parameters, e.g., arrival and service rates, but not about the actual queue length. 
 
\subsection{Motivation, objectives and contribution}\label{motivation-objectives}

The model of the present study is motivated by a number of different situations that arise in practice depending on whether the alternation between observable and unobservable periods is intentional or unintentional. A first such situation occurs when the information-providing mechanism has to respect some periodicity of the mode of operation of the system. For example, a given system may have to follow the day-night alternation and some of its features should be shut down in the night. An unintentional alternation between observable and unobservable periods occurs when the information-providing mechanism is unreliable. In this case, the system is observable as long as the mechanism operates properly, but it becomes unobservable when it fails and is under repair. Finally, a third case occurs when the administrator of the system intentionally stops to provide queue-length information for economic or other reasons, but resumes the information provision later. This is reasonable under various scenarios, e.g., if the information-providing mechanism is costly.\par
To understand how the information structure affects strategic customer behavior and system performance, we study these situations under the unified framework of a First-Come-First-Served (FCFS) M/M/1 queue that alternates between observable and unobservable periods. Upon arrival, customers decide whether to join or balk given the information that they receive from the system. In case that they enter the queue, they may renege (abandon the queue) any time later. Some basic considerations about the optimal equilibrium strategies are fairly straightforward in this setting. Customers that arrive at observable periods use Naor's \cite{Na1969} threshold strategy $n_e$: they join if the queue is less than a particular length and balk otherwise. Moreover, due to the FCFS discipline and the exponentially distributed service times, customers that enter during observable periods do not have an incentive to renege at any subsequent time. In contrast, customers that arrive during unobservable periods, will enter with some probability $q$ that depends on their expected payoff. Such customers may have an incentive to renege at the time of the first change to observable mode after their arrival instants. Regarding reneging, they again follow a threshold strategy with tolerance $n_s$ (in terms of queue length) that is at least as high as Naor's threshold. \par
Based on these considerations, the search for equilibrium strategies can be restricted in the class of strategies that are parametrized by two non-negative integer thresholds $n_e$ and $n_s$, and a probability $q$. We refer to such strategies as \emph{potentially equilibrium strategies} (PES) and denote them by $(n_e,n_s,q)$-(PES). Assuming that the population of customers adopts a PES, the state of the system can be described by a Continuous Time Markov Chain (CTMC). This enables the computation of the expected net benefit of a tagged customer that joins during an unobservable period in terms of the steady-state distribution of the CTMC and in turn, the equilibrium joining probability $q_e$. More specifically, using that the number of customers in the system is stochastically increasing in $q$ and that their net benefit is strictly decreasing in $q$, we derive  that such an equilibrium joining probability $q_e$ always exists, is unique and can be characterised via the system parameters. This concludes the technical analysis of the model and provides the required tractability for numerical experiments on its economic, operational and managerial aspects.\par
In this respect, our theoretical findings and experimental comparative statics indicate that the alternating information structure crucially affects system performance. To study the response of the equilibrium throughput and social welfare to the operational parameters of the system that can be tuned by a central decision maker, we perform two sets of numerical experiments (comparative statics): (i) in the fraction of time that the system remains unobservable and (ii) in the duration of unobservable periods.\par
By controlling the fraction of time that the system remains unobservable, we show that typically, a properly adjusted alternating system strictly improves over the continuously observable and continuously unobservable systems in terms of both equilibrium throughput and social welfare. This bridges the two cases that have been considered by Chen and Frank \cite{ChFr2004}, who show that the observable and unobservable systems are preferable for high and low arrival rates respectively. Regulating the fraction $\gamma$ of informed customers (by tuning the fraction time that the system is observable) has multiple effects on the system. In case of low arrival rates, an increase in $\gamma$ increases the customer entrance probability for the unobservable periods, but also increases the abandonment probability since the system passes quickly from the unobservable to the observable mode. This double effect is reversed in the case of high arrival rates. If reneging is forbidden and the alternation between observable and unobservable periods becomes very fast, then the present model reduces to the model of Hu, Li and Wang \cite{HuLiWa2018}. In this case, our results indicate a unimodal equilibrium throughput which agrees with the findings of \cite{HuLiWa2018}.\par
Turning to the second set of experiments, the control of the duration of the unobservable periods can be equivalently viewed as control of the rate of announcements that reveal the queue length to all present customers. The most interesting finding in this case is that the equilibrium throughput and typically also the social welfare are optimized for durations (announcement rates) strictly between $0$ and $\infty$. Decreasing the mean duration of unobservable periods between observable periods of fixed mean length (and hence, increasing the announcement rate) increases both the equilibrium joining probability and the reneging probability. The trade-off between the two effects is not clear and this is the reason for the unimodality of the throughput. Finally, in this setting and in the limiting case of very short observable periods, the present model reduces to the model of Burnetas, Economou and Vasiliadis \cite{BuEcVa2017}. The above results are in agreement with the findings of \cite{BuEcVa2017} which indicate that the equilibrium throughput is a non-monotonic function of the announcement rate.\par
A main finding of the numerical experiments concerns the diverse changes in the performance of the system that result from the complex interaction of its operational parameters. Their conflicting effects on main quantities, such as the joining probability, reneging probability or the expected customer net benefit, preclude general conclusions on equilibrium behavior over broad ranges of parameter values. For practical purposes, this implies that each different sets of parameters should be analyzed individually. However, it also underlines the importance of the derived equilibrium characterization that retains the functionality of the present model and which enables the comparative statics analysis via numerical methods and tools. 
 \par
We conclude our analysis with a third set of comparative statics on the economic parameters. Specifically, we study the effect on customer strategic behavior of the fraction of the entrance fee that is refundable and of the customers' service valuation. Our results show that the equilibrium throughput is an increasing or unimodal function in the refundable percentage. On the other hand, all performance measures are increasing both in the service valuation and in the service--entrance fee ratio for constant total (entrance and service) fee, as expected. \par
Our results imply that the alternating information structure of the present model unifies other existing approaches by obtaining as limiting cases both the heterogeneously observable model of \cite{HuLiWa2018} and the delayed observations model of \cite{BuEcVa2017}. Moreover, it significantly extends them: by fine-tuning the mean durations of the observable and unobservable periods, the present model can typically achieve higher equilibrium throughput and/or social welfare than the corresponding optimal instances of the models in \cite{BuEcVa2017} and \cite{HuLiWa2018} under the same operational and economic parameters. This comes at no cost in terms of tractability, since the current model always has a unique equilibrium customer strategy that can be characterised via well behaved economic quantities (customer net benefit). Hence, depending on the parameters that can be controlled in a given system or application, the present model is of practical relevance for its optimal design both from a social and a managerial perspective.

\subsection{Literature review}\label{sub:review}

Hassin \cite{Ha1986} compared the observable and the unobservable versions of the M/M/1 queue with strategic customers regarding their joining/balking dilemma, by focusing on the social welfare and a monopolist's profit under a profit-maximizing admission fee. Let $\lambda$ and $\mu$ denote respectively the arrival and service rate of a M/M/1 queue, $R$ be the service value and $C$ be the waiting cost per time unit for the customers. Hassin showed that if $R\mu\leq2 C$, then the profit under a profit-maximizing admission fee is larger for the observable model, for all $\lambda>0$. Hence, a profit maximizer prefers to reveal the queue length to the customers. If, however, $R\mu>2C$, then the profit under a profit-maximizing admission fee is larger for the observable model, if and only if $\lambda\geq\lambda^Z$ for some unique threshold arrival rate, $\lambda^Z$. Thus, in this case, a profit maximizer prefers to reveal the queue length only when $\lambda\geq\lambda^Z$. In other words, there is a range of the parameters ($R\mu>2C $ and $\lambda<\lambda^Z$), for which the provision of more information to the customers hurts the service provider. The same properties also hold for the social welfare under a profit-maximizing fee, but with a different critical value $\lambda^S$ in place of $\lambda^Z$. Thus, there is a range of the parameters ($\lambda<\lambda^S$), for which the provision of more information hurts the society as a whole. 

Chen and Frank \cite{ChFr2004} compared the observable and the unobservable versions of the M/M/1 queue by focusing on the equilibrium effective arrival rate (which is the same as the throughput since there are no abandonments), under an arbitrary fixed admission fee. For a given potential arrival rate $\lambda$, let $\lambda_e^{(o)}(\lambda)$ and $\lambda_e^{(u)}(\lambda)$ denote the corresponding equilibrium effective arrival rates in the observable and unobservable versions respectively. Chen and Frank proved that $\lambda_e^{(o)}(\lambda)-\lambda_e^{(u)}(\lambda)$ monotonically increases in $\lambda$ and there exists a critical value $\lambda^*$ such that $\lambda_e^{(o)}(\lambda^*)-\lambda_e^{(u)}(\lambda^*)=0$. Therefore, to attract more customers to the system, it is advisable to conceal the queue length for potential arrival rates $\lambda$ with $\lambda<\lambda^*$, and to reveal it when $\lambda>\lambda^*$.
Shone, Knight and Williams \cite{ShKnWi2013} considered the same problem of comparing the equilibrium throughputs, $\lambda_e^{(o)}$ and $\lambda_e^{(u)}$, between the observable and unobservable versions of the M/M/1 queue. They provided necessary and sufficient conditions on the system parameters under which the equilibrium effective arrival rates are equal in the two versions. Moreover, they investigated the behavior of the equilibrium effective arrival rates as functions of the normalized service value $\frac{R\mu}{C}$. In particular, they showed that the number of distinct normalized service values for which $\lambda_e^{(o)}=\lambda_e^{(u)}$ is monotonically increasing with respect to the utilization rate $\rho=\frac{\lambda}{\mu}$ and tends to infinity as $\rho\rightarrow 1$.

Guo and Zipkin \cite{GuZi2007} compared the observable, unobservable and workload observable versions of the M/M/1 queue, under a general reward-cost structure that generalizes the standard Naor's linear reward-cost structure. Under this framework, the service value is $R$, but a customer's waiting cost is $\theta E[c(W)]$, where $W$ stands for the waiting time, $c(w)$ is a common basic cost function for all customers, and $\theta$ is a customer-specific parameter that represents the sensitivity to delay. In other words, a customer with delay sensitivity $\theta$ has expected utility $R-\theta E[c(W)]$, if she decides to join. The authors showed that the maximum equilibrium throughput of the system is achieved at different information levels according to the values of the underlying parameters. Wang, Cui and Wang \cite{Wa18} consider an M/M/1 queueing system with a pay-for-priority option, and study customers’ joint decisions
between joining/balking and pay-for-priority. They compare the server’s revenue between the observable and the unobservable settings and interestingly, find that the service provider is better off with the observable setting when the system load is either low or high, but benefits more from the unobservable setting when the system load is medium.\par
The main conclusion from these studies is that the primary factor that determines whether information is good or bad for the service provider and the customers is the distribution function of the customer delay sensitivity and not the common basic cost function. As mentioned above, the relevant literature has reported three main approaches that aim to bridge the observable and unobservable versions of the M/M/1 queue: partially observable models, heterogeneously observable models and models with delayed observations.

In partially observable models, the state-space of the queue length of the M/M/1 queue is partitioned into subsets and the arriving customers are not informed about the exact queue length, but rather about the subset it belongs to. In other words, the waiting space can be considered to be `compartmented' and the customers are informed only about the compartment in which they are going to be placed.  Economou and Kanta \cite{EcKa2008} considered the case of regular compartmentalization (all compartments being of the same size), and showed that an ideal compartment-size exists, and it may be strictly between 1 (which corresponds to the observable version) and $\infty$ (which corresponds to the unobservable version).
Guo and Zipkin \cite{GuZi2009} considered the general case of compartments with possibly different sizes and proved several interesting results about the comparison of two partitions of the state-space, one a refinement of the other. More recently, Simhon, Hayel, Starobinski and Zhu \cite{SiHaStZh2016} considered the M/M/1 queue with strategic customers that face the dilemma of joining/balking when the administrator informs the customers about the current queue length only when it is short, i.e., when it does not exceed a certain threshold $D$. This corresponds to the partition of the state-space to the subsets $\{0\},\{1\},\{2\},\ldots,\{D\}$ and $\{D+1,D+2,\ldots\}$.  The authors proved that the equilibrium throughput is a monotone function of $D$ and hence, if the administrator's goal is to maximize throughput, then the optimal policy is one of the extremes, either the observable or the unobservable queue. Kim and Kim \cite{KiKi2017} considered the generalization of the last model, by assuming that the customers are informed about the current queue length only when it belongs to an arbitrary subset $O$. The authors proved the counter-intuitive result that the optimal partition for the maximization of the throughput of the system corresponds to a set $O$ that contains all the states above a threshold, i.e., it is preferable to allow the customers to observe the queue length only when it is large! Finally, Hassin and Koshman \cite{HaKo2014} considered a model where the arriving customers are only informed about whether the queue length is less than an exogenously given threshold $N$ or not. They focused on the profit maximization problem for the dynamic pricing version of this model (i.e., different prices are offered to the customers according to whether the queue length is below $N$ or not) and proved the interesting result that the choice of $N=1$ (meaning that customers are informed only if the server is idle) guarantees at least half of the maximum value that can be generated by the system. 

In heterogeneously observable models, the population of customers is divided into informed and uninformed customers, i.e., only a fraction of customers is allowed to observe the queue length before making their decisions. Two such models have been studied by Economou and Grigoriou \cite{EcGr2015} and Hu, Li and Wang \cite{HuLiWa2018}. Economou and Grigoriou determined the equilibrium strategies in the case where the service values and the waiting costs are different for informed and uninformed customers. Hu, Li and Wang proved that throughput and social welfare are in general unimodal and not monotonous in the fraction of informed customers. In other words, information heterogeneity in the population can lead to more efficient outcomes in terms of the system throughput or social welfare than information homogeneity. Moreover, they showed that for an overloaded system (with utilization factor sufficiently higher than 1), social welfare always reaches its maximum when some fraction of customers is uninformed.

In models with delayed observations, the customers decide whether to join or balk without knowing the state of the system, but later on they are informed about their current position and may renege. Burnetas, Economou and Vasiliadis \cite{BuEcVa2017} considered the M/M/1 queue where the administrator of the system makes periodic announcements to the customers about their current positions. The model was motivated by a situation that occurs when people submit petitions through certain web-based systems. Then, upon submission, the customers receive a confirmation message with the registration number of their petition. Later on, they learn the number of pending petitions in front of them. This is done either by periodic refreshments of a web page that indicates the registration number of the currently processed petition or by periodic bulk emails that announce the status of the pending petitions. The authors have shown that the equilibrium throughput is a non-monotonic function of the announcement rate. This implies that there exists an ideal announcement rate, strictly between 0 and $\infty$, that maximizes the throughput. In other words, some delay in providing information to the customers is beneficial in terms of throughput. Another model with delayed observation characteristics is the so-called `armchair decision' problem introduced by Hassin and Roet-Green \cite{HaRo2014} (see also Roet-Green \cite{Ro2013}). In this model, the customers observe the queue length before reaching it, using probably some web-based application. Then, they decide whether to leave their armchairs and go to the service facility or not, but when they arrive at the system, they are informed about the current queue length and should make their second decision, to join or balk. For more papers and thorough overviews of the results that concern the control of information in queueing systems with strategic customers, see Section 3.5, \textit{Information Control}, in Hassin \cite{Ha2016} and the recent review paper of Ibrahim \cite{Ib2018}.\par\vspace{0.2cm}
The rest of the paper is organized as follows: In Section \ref{model-strategies}, we define the model with the underlying reward-cost structure and derive sufficient conditions for customer equilibrium strategies. Section \ref{Performance-evaluation} presents our computations on system performance that lead to the characterization of equilibrium customer strategies in Section \ref{monotonicity-and-equilibrium}. In Section \ref{numerical-conclusions}, we perform a number of numerical experiments and conclude our analysis with some useful take-away messages and managerial insight in Section \ref{conclusions}. Some technical material is presented in Appendix \ref{generating-functions-appendix}. 

\section{Model and strategies}\label{model-strategies}

We consider an M/M/1 queue with arrival rate $\lambda$ and service rate $\mu$, where arriving customers decide whether to join or balk. The system alternates between unobservable and observable periods that are exponentially distributed with rates $\theta$ and $\zeta$, respectively. 
The customers are homogeneous in their valuations. More specifically, each one of them values service $R$ units and accumulates costs at rate $C$, as long as she stays in the system, either waiting or being served. The administrator of the system imposes an entrance fee $f_e$, which is paid by each customer who decides to join. Moreover, he imposes a service fee $f_s$ which is paid only by those customers that receive service. There is also a refund $r$ which is given to customers that decide to renege before their service has been completed. The refund may be positive (partial or full refund of the entrance fee) or negative (which means that there is a penalty for an abandonment) or even $-\infty$ (meaning that reneging is impossible or forbidden). 

We assume that the customers are strategic and risk neutral, in the sense that they want to maximize their expected net benefit, knowing that the others have similar objectives. During observable periods, the customers are informed about the queue length upon arrival and then make their joining/balking decisions, whereas in unobservable periods, they make their joining/balking decisions relying solely on their knowledge of the system parameters. All customers are informed about their current positions when the system enters an observable period and any time, they may decide to renege. 

To avoid trivial cases, we assume throughout the paper that the following two conditions hold:
\begin{equation}\label{R-big-enough}
R>f_e+f_s+\frac{C}{\mu},
\end{equation}
and
\begin{equation}\label{r-a-small-enough}
r\leq f_e.
\end{equation}

Condition \eqref{R-big-enough} ensures that the value of service is high enough so that a customer that observes an empty system prefers to join; otherwise the system would be continuously empty. 
\begin{table}[!htb]
\centering

\renewcommand{\arraystretch}{1.1} 
\begin{tabularx}{\linewidth}{rlX@{}}
\toprule
\multicolumn{3}{l}{\d\b Economic parameters}\\ 
$R>0$ && customer service valuation\\
$C>0$ && customer waiting cost per time unit in the system\\
$f_e\ge0$ && entrance fee\\
$f_s\ge0$ && service fee\\
$r\phantom{\hspace{2pt}\ge0}$&& $r\ge0$: \hspace{15pt}refund to reneging customers\\
&& $r<0$: \hspace{15pt}penalty for abandonment\\
&& $r=-\infty$: \hspace{1pt}reneging is forbidden \\[0.3cm]
\multicolumn{3}{l}{\d\b Operational parameters}\\
mode $1$ && observable system\\
mode $0$ && unobservable system\\
$\lambda>0$ && arrival rate\\
$\mu>0$ && service rate\\
$\theta>0$ && exponential rate of the duration of unobservable periods\\
$\zeta>0$ && exponential rate of the duration of observable periods \\
$B>0$ && mean duration of an information cycle: $B=1/\zeta+1/\theta$\\
$\gamma\in(0,1)$ && fraction of customers who arrive at observable mode: $\gamma=\theta/(\zeta+\theta)$\\
\multicolumn{1}{l}{$(n_e,n_s,q)$} && potential equilibrium strategy (PES)\\
$n_e\ge0$ && join/balk threshold (Naor's) at arrival instants during observable mode\\
$n_s\ge0$ && stay/renege threshold at change instants from unobservable to observable mode\\
$q\in [0,1]$ && joining probability in unobservable mode\\[0.3cm]
\multicolumn{3}{@{}l}{\d\b Performance measures}\\
$q_e\in[0,1]$ && equilibrium joining probability in unobservable mode\\
$\mu_e\ge0$ && equilibrium throughput: mean number of service completions per time unit\\
$S_e\ge0$ && equilibrium social welfare per time unit \\
\bottomrule
\end{tabularx}
\caption{Model parameters and notation}
\label{tab:notation}
\end{table}
Condition \eqref{r-a-small-enough} ensures that there are no customers that enter and remain in the system only instantaneously. Indeed, if $r>f_e$, then even a customer who observes a huge queue length upon arrival is willing to enter, but will renege immediately.\par

To analyze customer strategic behavior, we should take into account that the arriving customers face the dilemma of whether to join or balk upon arrival, and then, the joining customers continuously face the dilemma of whether to stay or renege. 

If a customer arrives at an observable period and finds $n$ customers in the system, then she enters if and only if $R-f_e-f_s-C\frac{n+1}{\mu}\geq 0$ or equivalently if

\begin{equation}\label{obs-join-threshold}
n+1\leq \left\lfloor \frac{\mu(R-f_e-f_s)}{C} \right\rfloor =n_e.
\end{equation}

Hence such a customer uses Naor's threshold and enters if her position in the system after joining is at most $n_e$, given by \eqref{obs-join-threshold}. If she enters, it is certain that she will stay in the system till her service completion, because no matter what the other customers do, her expected net benefit will be non-negative at all subsequent moments (because of the FCFS discipline and the Markovian framework). \par
If a customer arrives at an unobservable period, then she will enter with some probability $q$. If she enters, then she will certainly stay till the first time that the system becomes observable or till her service completion, whatever occurs first (again because of the assumption of exponentially distributed times). Suppose that the system becomes observable before the service completion of a customer that joined in an unobservable period. If $n$ is the current position of the customer, then she will stay if and only if $R-f_s-C\frac{n}{\mu}\geq r$ or equivalently if

\begin{equation}\label{unobs-to-obs-stay-threshold}
n\leq \left\lfloor \frac{\mu(R-r-f_s)}{C} \right\rfloor =n_s.
\end{equation}
Note that $f_e$ does not play any role for this decision, since it is refundable only to the extent specified by $r$. 
Of course, if the tagged customer decides to stay after the moment that the system becomes observable, then she will stay till her service completion, as her expected benefit cannot deteriorate. Note, also, that 
$$n_e\leq n_s$$
because of condition \eqref{r-a-small-enough}. From the above discussion, we conclude that an equilibrium strategy should
\begin{itemize}
\item prescribe \emph{enter} if the offered position for an arriving customer during an observable period is at most $n_e$ given by \eqref{obs-join-threshold},
\item prescribe \emph{enter} with some probability $q$ if a customer arrives during an unobservable period, and
\item prescribe \emph{stay} if the position of a customer at the time of a change from an unobservable to observable mode is at most $n_s$ given by \eqref{unobs-to-obs-stay-threshold}.
\end{itemize}
Thus, a \emph{potential equilibrium strategy} (PES) should satisfy the above three conditions and will be referred to as an $(n_e,n_s,q)$-PES. Note that only $q$ remains to be determined. All model parameters are summarized for convenience in Table~\ref{tab:notation}.

\section{Customer expected net benefit}\label{Performance-evaluation}

Suppose that the population of the customers adopts an $(n_e,n_s,q)$-PES and consider a tagged customer. Her best response against the $(n_e,n_s,q)$-PES will necessarily be an $(n_e,n_s,q')$-PES, with a possibly different joining probability $q'$. This follows from the discussion in Section \ref{model-strategies}. To compute the best response of the tagged customer, suppose that she arrives during an unobservable period and let $\mathcal{U}(n_e,n_s,q)$ be her expected net benefit if the other customers follow the $(n_e,n_s,q)$-PES and the tagged customer decides to join and will use the same threshold $n_s$ for staying/reneging as the other customers at the time where the system becomes observable. 

To determine $\mathcal{U}(n_e,n_s,q)$, we condition on the number of customers $N^-_q$ that are present in the system just before the arrival of the tagged customer (noting that $N^-_q$  is not observable for her). Then,
\begin{equation}\label{unconditional-benefit-01}
\mathcal{U}(n_e,n_s,q)=\sum_{n=0}^{\infty} \Pr[N^-_q=n] U(n;n_s),
\end{equation}
where $U(n;n_s)$ stands for the conditional expected net benefit of the tagged customer if she decides to join, given that $n$ customers are present in the system upon her arrival (excluding herself) and all of them use the same reneging threshold $n_s$. Note that $U(n;n_s)$ is independent from the strategy parameters $n_e$, $q$, and the operational model parameters $\lambda$, $\zeta$. Also, the subscript $q$ in $N^{-}_q$ is present to emphasize the dependence on $q$ which is crucial in several derivations. To evaluate $\mathcal{U}(n_e,n_s,q)$ using \eqref{unconditional-benefit-01}, we start with computing the conditional expected values $U(n;n_s)$. 

\begin{thm}\label{conditional-expected-net-benefit-thm}
The conditional expected net benefit of an arriving customer if she decides to join, given that $n$ customers are present in the system upon her arrival and all of them use the reneging threshold $n_s$ is given by
\begin{equation}\label{conditional-expected-net-benefit-formula}
U(n;n_s)=\left\{
\begin{array}{ll}
R-f_e-f_s-\frac{C(n+1)}{\mu} & \mbox{if } 0\leq n < n_s\\
r-f_e-\frac{C}{\theta}+( R-r-f_s-\frac{Cn_s}{\mu}+\frac{C}{\theta})(\frac{\mu}{\mu+\theta})^{n+1-n_s} & \mbox{if } n \geq n_s.
\end{array}
\right.
\end{equation} 
\end{thm}

\begin{proof} Consider a tagged customer that arrives and decides to join during an unobservable period. Suppose that the system has $n$ other customers at that moment. We consider two cases according to whether $n< n_s$ or not.

In the first case, where $n< n_s$, the customer will certainly complete her service. Therefore, she will receive the service reward $R$ and will pay the entrance and service fees, $f_e$ and  $f_s$. Moreover, her sojourn time in the system will be the sum of $n+1$ service times and we conclude with the first branch of \eqref{conditional-expected-net-benefit-formula}.

In the second case, where the tagged customer arrives when there are $n\geq n_s$ customers and decides to join, let $X$ be the time till the beginning of the first observable period after her arrival and $K$ be the number of service completions that occur during $X$. Let $U_n$ be the net benefit of the tagged customer. Then, by conditioning on $X$ and $K$ we have
\begin{equation}\label{mathcal-U-appendix}
U(n;n_s)=\int_0^{\infty} \sum_{k=0}^{\infty}E[U_n\mid X=x,K=k]\Pr[K=k\mid X=x]f_X(x)dx,
\end{equation}
where $f_X(x)$ is the probability density of $X$. 
Note that $f_X(x)=\theta e^{-\theta x}$, $x>0$, because of the memoryless property of the exponential distribution, and that $\Pr[K=k\mid X=x]=e^{-\mu x} \frac{(\mu x)^k}{k!}$, $k\geq 0$, assuming that the server is always busy during time $X$. Moreover,  the conditional mean value $E[U_n\mid X=x,K=k]$ is given by
\begin{equation}\label{conditional-U-n-given-X-K}
E[U_n\mid X=x,K=k]=\left\{
\begin{array}{ll}
-f_e-Cx+r, & \mbox{if } 0\leq k\leq n-n_s,\\
-f_e-C\(x+\frac{n+1-k}{\mu}\)+R-f_s, & \mbox{if } n-n_s+1\leq k \leq n,\\
-f_e-C\frac{(n+1)x}{k+1}+R-f_s, & \mbox{if } k\geq n+1.
\end{array}
\right.
\end{equation}
Indeed, to justify \eqref{conditional-U-n-given-X-K}, we notice that there are three cases regarding the tagged customer:
\begin{itemize}
\item[Case I:] At the beginning of the first observable period after her arrival, the tagged customer occupies a position greater than $n_s$.

In this case, the number of service completions $k$ is such that $n+1-k\geq n_s+1$, i.e., $k\leq n-n_s$. The tagged customer will renege at time $x$, so her sojourn time in the system is  the time till the first announcement $x$. Therefore, she will pay $f_e$ upon entrance, suffer waiting cost $Cx$ and receive the refund $r$ upon abandonment. So, we conclude the first branch of \eqref{conditional-U-n-given-X-K}.

\item[Case II:] At the beginning of the first observable period after her arrival, the tagged customer occupies a position less than or equal to $n_s$.

In this case, the number of service completions $k$ is such that $1\leq n+1-k \leq n_s$, i.e., $n-n_s+1\leq k \leq n$. The tagged customer will not renege and has to wait for the completion of $n+1-k$ service times after the beginning of the observable period that followed her arrival. Her mean sojourn time in the system will be $x+\frac{n+1-k}{\mu}$. Therefore, she will suffer waiting cost $C(x+\frac{n+1-k}{\mu})$ and the net service value will be $R-f_e-f_s$, so we conclude the second branch of \eqref{conditional-U-n-given-X-K}.
\item[Case III:] At the beginning of the first observable period after her arrival, the tagged customer has already been served.

In this case, the number of service completions $k$ is such that $k\geq n+1$. In the interval of length $x$, till the beginning of the observable period, there were $k$ events in the Poisson process of the service completions and the departure of the tagged customer corresponds to the $(n+1)$-th event of these $k$ events. But then, the conditional distribution of the departute time of the tagged customer given $x$ coincides with the $(n+1)$-order statistic of $k$ i.i.d. uniform random variables in $[0,x]$ (see e.g., Campbell's Theorem in Section 5.3 of Kulkarni \cite{Ku2010}). Its mean value is $\frac{(n+1)x}{k+1}$, so the waiting cost for the tagged customer is $C\frac{(n+1)x}{k+1}$, while the net service value is $R-f_e-f_s$. We conclude with the third branch of \eqref{conditional-U-n-given-X-K}.
\end{itemize}
We can now insert \eqref{conditional-U-n-given-X-K} into \eqref{mathcal-U-appendix} and we obtain
\begin{align*}
U(n;n_s)=&\int_0^{\infty}\sum_{k=0}^{n-n_s} (-f_e-Cx+r) e^{-\mu x} \frac{(\mu x)^k}{k!} \theta e^{-\theta x}dx\\
& +\int_0^{\infty}\sum_{k=n-n_s+1}^{n} (-f_e-Cx-C\frac{n+1-k}{\mu}+R-f_s) e^{-\mu x} \frac{(\mu x)^k}{k!} \theta e^{-\theta x}dx\nonumber\\
& +\int_0^{\infty}\sum_{k=n+1}^{\infty} (-f_e-C\frac{(n+1)x}{k+1}+R-f_s) e^{-\mu x} \frac{(\mu x)^k}{k!} \theta e^{-\theta x}dx.
\end{align*}
Evaluating the integrals (using the formula $\int_0^{\infty} x^he^{-\nu x}dx=\frac{h!}{\nu^{h+1}}$) and grouping similar terms yields
\begin{align*}
U(n;n_s)=&-f_e+r\sum_{k=0}^{n-n_s} \frac{\mu^k \theta}{(\mu+\theta)^{k+1}} + (R-f_s)\sum_{k=n-n_s+1}^{\infty} \frac{\mu^k \theta}{(\mu+\theta)^{k+1}}\nonumber\\
& -C\frac{1}{\mu}\sum_{k=1}^{n-n_s} \frac{k\mu^k\theta}{(\mu+\theta)^{k+1}}-C\frac{n+1}{\mu} \sum_{k=n-n_s+1}^{\infty} \frac{\mu^k\theta}{(\mu+\theta)^{k+1}},\; n\geq n_s.
\end{align*}
Evaluating the geometric sums and grouping equal terms yields after some simplifications the second branch of \eqref{conditional-expected-net-benefit-formula}.
\end{proof}\vspace{0.3cm}

To proceed with the computation of $\mathcal{U}(n_e,n_s,q)$ using \eqref{unconditional-benefit-01}, we need to compute the probabilities $\Pr[N^-_q=n]$. These probabilities can be computed by studying the dynamics of the system, when the population of customers follow the $(n_e,n_s,q)$-PES. Indeed, under this strategy, the state of the system is described by a continuous-time Markov chain (CTMC) $\{(N(t),I(t)):t\geq 0\}$, where $N(t)$ records the number of customers in the system at time $t$, while $I(t)$ denotes the mode of operation ($I(t)=1$ during observable periods and $I(t)=0$ during unobservable periods). The state-space of $\{(N(t),I(t))\}$ is $\mathcal{S}_{N,I}=\{(n,0):n\geq 0\}\cup\{(n,1):0\leq n\leq n_s\}$ and the corresponding transition diagram is shown in Figure \ref{Figure01}. 

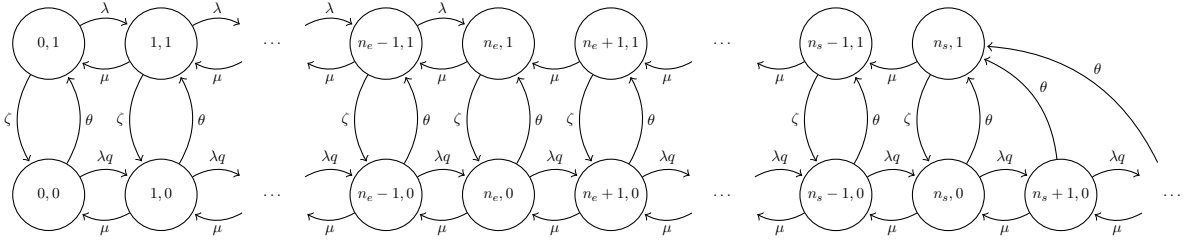
\begin{figure}[!htb]\centering
\begin{scaletikzpicturetowidth}{\linewidth}
\begin{tikzpicture}[state/.style={circle, draw, minimum size=1.8cm}, scale=\tikzscale, every node/.append style={transform shape}]

\node[state] 					(01) {$0,1$};
\node[state, right=of 01] 		(11) {$1,1$};

\node[circle, minimum size=1.8cm, right=of 11] 	(dots1) {$\dots$};
\node[state, right=of dots1] 	(ne-11) {$n_e-1,1$};
\node[state, right=of ne-11] 	(ne1) {$n_e,1$};
\node[state, right=of ne1] 		(ne+11) {$n_e+1,1$};

\node[circle, minimum size=1.8cm, right=of ne+11] 	(dots2) {$\dots$};
\node[state, right=of dots2] 	(ns-11) {$n_s-1,1$};
\node[state, right=of ns-11] 	(ns1) {$n_s,1$};

\node[state, below=2cm of 01] (00) {$0,0$};
\node[state, right=of 00] 		(10) {$1,0$};

\node[circle, minimum size=1.8cm, right=of 10] 	(dots3) {$\dots$};
\node[state, right=of dots3] 	(ne-10) {$n_e-1,0$};
\node[state, right=of ne-10] 	(ne0) {$n_e,0$};
\node[state, right=of ne0] 	(ne+10) {$n_e+1,0$};

\node[circle, minimum size=1.8cm, right=of ne+10] 	(dots4) {$\dots$};
\node[state, right=of dots4] 	(ns-10) {$n_s-1,0$};
\node[state, right=of ns-10] 	(ns0) {$n_s,0$};
\node[state, right=of ns0] 	(ns+10) {$n_s+1,0$};
\node[circle, minimum size=1.8cm, right=of ns+10] 	(dots5) {$\dots$};

\draw[every loop, auto=left, bend left] 
(01) 		edge node {$\lambda$} 		(11)
(11) 		edge node {$\lambda$} 		(dots1)
(dots1) 		edge node {$\lambda$} 		(ne-11)
(ne-11) 		edge node {$\lambda$} 		(ne1)

(00) 		edge node {$\lambda q$} 	(10)
(10) 		edge node {$\lambda q$} 	(dots3)
(dots3) 		edge node {$\lambda q$} 	(ne-10)
(ne-10) 		edge node {$\lambda q$} 	(ne0)
(ne0) 		edge node {$\lambda q$} 	(ne+10)
(ne+10)		edge node {$\lambda q$} 	(dots4)
(dots4) 		edge node {$\lambda q$} 	(ns-10)
(ns-10) 		edge node {$\lambda q$} 	(ns0)
(ns0) 		edge node {$\lambda q$} 	(ns+10)
(ns+10) 		edge node {$\lambda q$} 	(dots5)

(11) 		edge node {$\mu$} 			(01)
(dots1) 		edge node {$\mu$} 			(11)
(ne-11) 		edge node {$\mu$} 			(dots1)
(ne1) 		edge node {$\mu$} 			(ne-11)
(ne+11) 	edge node {$\mu$} 			(ne1)
(dots2) 		edge node {$\mu$} 			(ne+11)
(ns-11) 		edge node {$\mu$} 			(dots2)
(ns1) 		edge node {$\mu$} 			(ns-11)

(10) 		edge node {$\mu$} 			(00)

(dots3) 		edge node {$\mu$} 			(10)
(ne-10) 		edge node {$\mu$} 			(dots3)
(ne0) 		edge node {$\mu$} 			(ne-10)
(ne+10) 	edge node {$\mu$} 			(ne0)
(dots4) 		edge node {$\mu$} 			(ne+10)
(ns-10) 		edge node {$\mu$} 			(dots4)
(ns0) 		edge node {$\mu$} 			(ns-10)
(ns+10) 		edge node {$\mu$} 			(ns0)
(dots5) 		edge node {$\mu$} 			(ns+10);

\draw[every loop, auto=right, bend right]
(01) 		edge node {$\zeta$}	(00)
(11) 		edge node {$\zeta$}	(10)

(ne-11) 		edge node {$\zeta$}	(ne-10)
(ne1) 		edge node {$\zeta$}	(ne0)
(ne+11) 	edge node {$\zeta$}	(ne+10)
(ns-11) 		edge node {$\zeta$}	(ns-10)
(ns1) 		edge node {$\zeta$}	(ns0)
(00) 		edge node {$\theta$} (01)
(10) 		edge node {$\theta$} (11)

(ne-10) 		edge node {$\theta$} (ne-11)
(ne0) 		edge node {$\theta$} (ne1)
(ne+10) 	edge node {$\theta$} (ne+11)
(ns-10) 		edge node {$\theta$} (ns-11)
(ns0) 		edge node {$\theta$} (ns1)
(ns+10) 		edge node {$\theta$} (ns1)
(dots5) 		edge node {$\theta$} (ns1);

\end{tikzpicture}
\end{scaletikzpicturetowidth}
\caption{Transition diagram of the system state, when the customers follow the $(n_e,n_s,q)$-PES.}\label{Figure01}
\end{figure}\noindent
Its transition rates are given by
\begin{equation*}
q_{(n,i)(m,j)}=\left\{
\begin{array}{lll}
\lambda q & \mbox{if } i=j=0, & n\geq 0,\;\; m=n+1,\\
\mu & \mbox{if } i=j=0, & n\geq 1,\;\; m=n-1,\\
\lambda & \mbox{if } i=j=1, & 0\leq n\leq n_e-1,\;\; m=n+1,\\
\mu & \mbox{if } i=j=1, & 1\leq n\leq n_s,\;\; m=n-1,\\
\theta & \mbox{if } i=0, j=1, & 0\leq n\leq n_s-1, \;\; m=n,\\
\theta & \mbox{if } i=0, j=1, & n\geq n_s, \;\; m=n_s,\\
\zeta & \mbox{if } i=1, j=0, & 0\leq n\leq n_s, \;\; m=n,\\
0 & \mbox{otherwise}.
\end{array}
\right.
\end{equation*}
Denote by $(p(n,i): (n,i)\in\mathcal{S}_{N,I})$ its steady-state distribution. Then,
\begin{equation}\label{embedded-distribution}
\Pr[N^-_q=n]=\frac{p(n,0)}{\sum_{m=0}^{\infty} p(m,0)}.
\end{equation} 
Therefore, we need to compute $(p(n,i): (n,i)\in\mathcal{S}_{N,I})$. The distribution $(p(n,i))$ is the unique normalized solution of the following balance equations:
\begin{eqnarray}
(\lambda+\zeta) p(0,1) &=& \theta p(0,0) +\mu p (1,1),\label{bal-1}\\
(\lambda+\mu+\zeta) p(n,1) &=&\lambda p (n-1,1)+\theta p(n,0)+\mu p(n+1,1),\;\;\; 1\leq n\leq n_e-1,\label{bal-2}\\
(\mu+\zeta) p(n_e,1)&=&\lambda p(n_e-1,1)+\theta p(n_e,0)+\mu p(n_e+1,1),\label{bal-3}\\
(\mu+\zeta)p(n,1) &=&\theta p(n,0)+\mu p(n+1,1),\;\;\; n_e+1\leq n \leq n_s-1,\label{bal-4}\\
(\mu+\zeta) p(n_s,1)&=&\theta\sum_{n=n_s}^{\infty}p(n,0),\label{bal-5}\\
(\lambda q+\theta) p(0,0) &=& \zeta p (0,1)+\mu p(1,0),\label{bal-6}\\
(\lambda q+\mu+\theta) p(n,0)&=& \lambda q p(n-1,0)+\zeta p(n,1)+\mu p(n+1,0),\;\;\; 1\leq n\leq n_s,\label{bal-7}\\
(\lambda q+\mu+\theta) p(n,0)&=& \lambda q p(n-1,0)+\mu p(n+1,0),\;\;\; n\geq n_s+1.\label{bal-8}
\end{eqnarray}
Solving the system of the balance equations and the normalization equation 
\begin{equation}
\label{norm_cond}\sum_{(n,i)\in\mathcal{S}_{N,I}}p(n,i)=1,
\end{equation} 
and using \eqref{unconditional-benefit-01}, \eqref{conditional-expected-net-benefit-formula} and \eqref{embedded-distribution}, we can compute $\mathcal{U}(n_e,n_s,q)$ which is the key quantity for deriving the customer equilibrium behavior. The system of the balance equations and the normalization equation is quite involved. However, it can be reduced easily to a finite linear system because of the homogeneous 1-dimensional nature of the model for $n\geq n_s+1$. Indeed, \eqref{bal-8} is a homogeneous linear second-order difference equation. Therefore, its solution is
\begin{equation}\label{p-n-0-general}
p(n,0)=c_+\rho_+^n+c_-\rho_-^n,\;\;\; n\geq n_s,
\end{equation}
where $\rho_+$ and $\rho_-$ are the roots of the corresponding characteristic equation
\begin{equation}
\mu x^2-(\lambda q+\mu+\theta)x+\lambda q=0\label{characteristic-equation}
\end{equation}
and $c_+$, $c_-$ constants to be determined (see e.g., Section 2.3 in Elaydi \cite{El1996}). Therefore,
\begin{equation}
\rho_{+,-}=\frac{\lambda q +\mu +\theta\pm \sqrt{(\lambda q +\mu +\theta)^2-4\lambda q \mu}}{2\mu}.
\end{equation}
It is easy now to see that $\rho_+>1$, whereas $0<\rho_-<1$. Because of the normalization equation, $\sum_{n=n_s}^{\infty} p(n,0)<\infty$, so necessarily the coefficient $c_+$ of $\rho_+$ in \eqref{p-n-0-general} should be 0. Hence, \eqref{p-n-0-general} becomes
\begin{equation}\label{p-n-0-general-2}
p(n,0)=c_-\rho_-^n=\( \frac{\lambda q +\mu +\theta - \sqrt{(\lambda q +\mu +\theta)^2-4\lambda q \mu}}{2\mu}\)^np(n_s,0),\;\;\; n\geq n_s.
\end{equation}
Now, equation \eqref{bal-5} is written as 
\begin{equation}
(\mu+\zeta) p(n_s,1)=\frac{\theta}{1-\rho_{-}}p(n_s,0),\label{bal-5'}
\end{equation}
and equation \eqref{bal-7} for $n=n_s$ reduces to 
\begin{equation*}
(\lambda q+\mu+\theta) p(n_s,0)= \lambda q p(n_s-1,0)+\zeta p(n_s,1)+\mu\rho_{-} p(n_s,0).
\end{equation*}
The latter can be also written as
\begin{equation}
( \frac{\theta}{1-\rho_{-}}+\mu) p(n_s,0)=\zeta p(n_s,1)+\lambda q p(n_s-1,0), \label{bal-7'}
\end{equation}
using that $\rho_{-}$ satisfies \eqref{characteristic-equation}. Now, equations \eqref{bal-1}-\eqref{bal-4}, \eqref{bal-5'}, \eqref{bal-6}, \eqref{bal-7} for $1\leq n\leq n_s-1$  and \eqref{bal-7'} show that $(p(n,i): i=0,1 \mbox{ and } 0\leq n\leq n_s)$ satisfies the balance equations of the finite non-homogeneous Quasi-Birth-Death (QBD) process that results from the original chain in Figure \ref{Figure01}, when the states $(n,0)$ for $n\geq n_s+1$ are omitted and the rate from $(n_s,0)$ to $(n_s,1)$ becomes $\theta/(1-\rho_{-})$. Indeed, this finite QBD process, with transition rate diagram given in Figure \ref{Figure02}, is the censored process that results from the original chain observed only while it stays in states with $n\leq n_s$ (for details see \cite{LaRa1999}, Chapter 5, in particular Section 5.5). Hence, $(p(n,i): i=0,1 \mbox{ and } 0\leq n\leq n_s)$ can be effectively computed up to a normalization constant by using any general algorithm for the computation of the steady-state distributions of finite QBD processes (see e.g., \cite{LaRa1999}, Chapter 12 or \cite{ArGo2008}, Section 7.2.1). Then, the remaining steady-state probabilities $p(n,0)$, $n\geq n_s+1$, are computed up to the same normalization constant by \eqref{p-n-0-general-2}. Finally, the normalization constant is computed using the normalization equation \eqref{norm_cond}. Note that the algorithms for the computation of the steady-state probabilities of a finite QBD are very fast since they are of block-Gaussian elimination type that exploit the block-tridiagonal form of the transition rate matrix. 

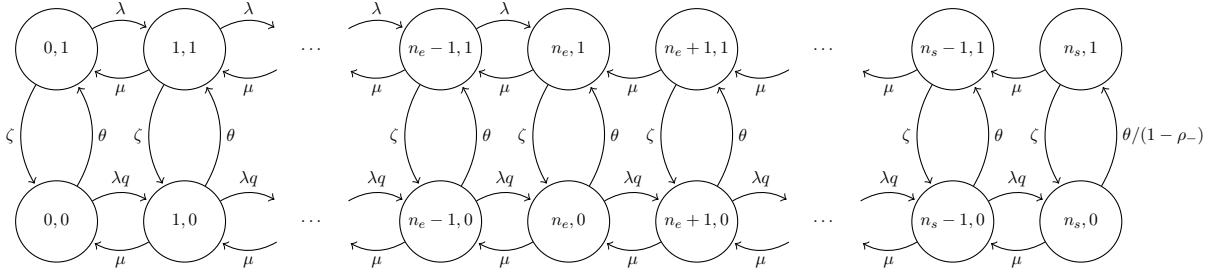
\begin{figure}[!htb]\centering
\begin{scaletikzpicturetowidth}{\textwidth}
\begin{tikzpicture}[state/.style={circle, draw, minimum size=1.8cm}, scale=\tikzscale, every node/.append style={transform shape}]

\node[state] 					(01) {$0,1$};
\node[state, right=of 01] 		(11) {$1,1$};

\node[circle, minimum size=1.8cm, right=of 11] 	(dots1) {$\dots$};
\node[state, right=of dots1] 	(ne-11) {$n_e-1,1$};
\node[state, right=of ne-11] 	(ne1) {$n_e,1$};
\node[state, right=of ne1] 		(ne+11) {$n_e+1,1$};

\node[circle, minimum size=1.8cm, right=of ne+11] 	(dots2) {$\dots$};
\node[state, right=of dots2] 	(ns-11) {$n_s-1,1$};
\node[state, right=of ns-11] 	(ns1) {$n_s,1$};

\node[state, below=2cm of 01] (00) {$0,0$};
\node[state, right=of 00] 		(10) {$1,0$};

\node[circle, minimum size=1.8cm, right=of 10] 	(dots3) {$\dots$};
\node[state, right=of dots3] 	(ne-10) {$n_e-1,0$};
\node[state, right=of ne-10] 	(ne0) {$n_e,0$};
\node[state, right=of ne0] 	(ne+10) {$n_e+1,0$};

\node[circle, minimum size=1.8cm, right=of ne+10] 	(dots4) {$\dots$};
\node[state, right=of dots4] 	(ns-10) {$n_s-1,0$};
\node[state, right=of ns-10] 	(ns0) {$n_s,0$};

\draw[every loop, auto=left, bend left] 
(01) 		edge node {$\lambda$} 		(11)
(11) 		edge node {$\lambda$} 		(dots1)
(dots1) 		edge node {$\lambda$} 		(ne-11)
(ne-11) 		edge node {$\lambda$} 		(ne1)

(00) 		edge node {$\lambda q$} 	(10)
(10) 		edge node {$\lambda q$} 	(dots3)
(dots3) 		edge node {$\lambda q$} 	(ne-10)
(ne-10) 		edge node {$\lambda q$} 	(ne0)
(ne0) 		edge node {$\lambda q$} 	(ne+10)
(ne+10)		edge node {$\lambda q$} 	(dots4)
(dots4) 		edge node {$\lambda q$} 	(ns-10)
(ns-10) 		edge node {$\lambda q$} 	(ns0)

(11) 		edge node {$\mu$} 			(01)
(dots1) 		edge node {$\mu$} 			(11)
(ne-11) 		edge node {$\mu$} 			(dots1)
(ne1) 		edge node {$\mu$} 			(ne-11)
(ne+11) 	edge node {$\mu$} 			(ne1)
(dots2) 		edge node {$\mu$} 			(ne+11)
(ns-11) 		edge node {$\mu$} 			(dots2)
(ns1) 		edge node {$\mu$} 			(ns-11)

(10) 		edge node {$\mu$} 			(00)
(dots3) 		edge node {$\mu$} 			(10)
(ne-10) 		edge node {$\mu$} 			(dots3)
(ne0) 		edge node {$\mu$} 			(ne-10)
(ne+10) 	edge node {$\mu$} 			(ne0)
(dots4) 		edge node {$\mu$} 			(ne+10)
(ns-10) 		edge node {$\mu$} 			(dots4)
(ns0) 		edge node {$\mu$} 			(ns-10)
;

\draw[every loop, auto=right, bend right]
(01) 		edge node {$\zeta$}	(00)
(11) 		edge node {$\zeta$}	(10)

(ne-11) 		edge node {$\zeta$}	(ne-10)
(ne1) 		edge node {$\zeta$}	(ne0)
(ne+11) 	edge node {$\zeta$}	(ne+10)
(ns-11) 		edge node {$\zeta$}	(ns-10)
(ns1) 		edge node {$\zeta$}	(ns0)
(00) 		edge node {$\theta$} (01)
(10) 		edge node {$\theta$} (11)

(ne-10) 		edge node {$\theta$} (ne-11)
(ne0) 		edge node {$\theta$} (ne1)
(ne+10) 	edge node {$\theta$} (ne+11)
(ns-10) 		edge node {$\theta$} (ns-11)
(ns0) 		edge node {$\theta/(1-\rho_{-})$} (ns1)
;

\end{tikzpicture}
\end{scaletikzpicturetowidth}
\caption{Transition diagram of the censored process in the set of states with $n\leq n_s$.}\label{Figure02}
\end{figure}\noindent
We now proceed to obtain a formula for $\mathcal{U}(n_e,n_s,q)$. By \eqref{unconditional-benefit-01} and \eqref{embedded-distribution}, we have that 
\begin{eqnarray}
\mathcal{U}(n_e,n_s,q)&=& \sum_{n=0}^{\infty} \frac{p(n,0)}{\sum_{m=0}^{\infty}p(m,0)}U(n;n_s)=\frac{\zeta+\theta}{\zeta}\sum_{n=0}^{\infty} p(n,0) U(n;n_s),\label{mathcal-U-proof-1}
\end{eqnarray}
because $\{I(t)\}$ is a 2-state CTMC with rate $q_{01}=\theta$ and $q_{10}=\zeta$ and so we conclude that $\sum_{m=0}^{\infty}p(m,0)=\frac{\zeta}{\zeta+\theta}$. To evaluate the sum in \eqref{mathcal-U-proof-1}, we decompose it in two sums according to the two branches of \eqref{conditional-expected-net-benefit-formula}. We have:
\begin{align}
\sum_{n=0}^{n_s-1} p(n,0) U(n;n_s)&=\sum_{n=0}^{n_s-1} p(n,0) \left( R-f_e-f_s-\frac{C(n+1)}{\mu} \right),\label{sum-1-uncond-exp-ben}
\end{align}
\vspace*{-1cm}
\begin{align}
\sum_{n=n_s}^{\infty} p(n,0) U(n;n_s)&=\sum_{n=n_s}^{\infty} p(n,0)\left( r_a-f_e-\frac{C}{\theta} \right)\nonumber\\
&+\sum_{n=n_s}^{\infty} p(n,0) \left( R-r_a-f_s-\frac{Cn_s}{\mu}+\frac{C}{\theta} \right) \left( \frac{\mu}{\mu+\theta} \right)^{n+1-n_s}.\label{sum-3-uncond-exp-ben}
\end{align}
The right-hand sides of \eqref{sum-1-uncond-exp-ben}-\eqref{sum-3-uncond-exp-ben} can be written more compactly in terms of the following partial generating functions of the steady-state distribution $(p(n,i): (n,i)\in\mathcal{S}_{N,I})$, that  correspond to the various groups of states that appear in the transition diagram (except from the state $(n_s,1)$ which forms a group by itself):
\begin{align}
P_{0a}(z)&=\sum_{n=0}^{n_e-1} p(n,0)z^n,\;\; P_{0b}(z)=\sum_{n=n_e}^{n_s-1} p(n,0)z^{n-n_e},\;\;
P_{0c}(z)=\sum_{n=n_s}^{\infty} p(n,0)z^{n-n_s}\label{gener-functions-0}\\
P_{1a}(z)&=\sum_{n=0}^{n_e-1} p(n,1)z^n,\;\; P_{1b}(z)=\sum_{n=n_e}^{n_s-1} p(n,1)z^{n-n_e}.\label{gener-functions-1}
\end{align}
Indeed, equations  \eqref{sum-1-uncond-exp-ben}-\eqref{sum-3-uncond-exp-ben} assume the form
\begin{align}
\sum_{n=0}^{n_s-1} p(n,0) U(n;n_s)
=&\left( R-f_e-f_s-\frac{C}{\mu} \right)P_{0a}(1)-\frac{C}{\mu}P_{0a}'(1)\nonumber\\
& \;\;\;+\left( R-f_e-f_s-\frac{C}{\mu} \right)P_{0b}(1)-\frac{C}{\mu}P_{0b}'(1)-\frac{Cn_e}{\mu}P_{0b}(1),\label{sum-1-uncond-exp-ben-gf}\\
\sum_{n=n_s}^{\infty} p(n,0) U(n;n_s)
=&\left( r_a-f_e-\frac{C}{\theta} \right) P_{0c}(1)\nonumber\\
&\;\;\;+\left( R-r_a-f_s-\frac{Cn_s}{\mu}+\frac{C}{\theta} \right) \frac{\mu}{\mu+\theta}P_{0c}\left( \frac{\mu}{\mu+\theta} \right).\label{sum-3-uncond-exp-ben-gf}
\end{align}
Combining \eqref{mathcal-U-proof-1} with \eqref{sum-1-uncond-exp-ben-gf}-\eqref{sum-3-uncond-exp-ben-gf} yields an expression for $\mathcal{U}(n_e,n_s,q)$ which is reported in the following Theorem.
\begin{thm}\label{thm:unconditional}
The unconditional expected net benefit of an arriving customer if she decides to join, given that the population of the customers follow the $(n_e,n_s,q)$-PES is given by
\begin{align*}
\mathcal{U}(n_e,n_s,q)=\frac{\zeta+\theta}{\zeta}&\left[(R-f_e-f_s-\frac{C}{\mu}\right) \left(P_{0a}(1)+P_{0b}(1))-\frac{Cn_e}{\mu}P_{0b}(1) \right.\\
&\;\;\;\;+\left(r_a-f_e-\frac{C}{\theta}\right) P_{0c}(1) -\frac{C}{\mu} \left(P_{0a}'(1)+P_{0b}'(1)\right)\\
&\;\;\;\;\left. +\left( R-r_a-f_s-\frac{Cn_s}{\mu}+\frac{C}{\theta}\right)\frac{\mu}{\mu+\theta}P_{0c}\left( \frac{\mu}{\mu+\theta}\right)\right].
\end{align*} 
\end{thm}
\noindent
The probability generating functions that are needed to evaluate $\mathcal{U}(n_e,n_s,q)$ can be computed either using the steady-state probabilities that are obtained via the finite QBD approach that we described above, or directly using a generating function approach. The latter approach is interesting and efficient, but quite involved. So we describe it in detail in the Appendix of the paper. Because of its complexity, it may be preferable than the former approach only when $n_s$ is large, in which case the finite QBD approach is computationally costly. Having determined the unconditional expected net benefit function $\mathcal{U}(n_e,n_s,q)$, we can now proceed towards the characterization and computation of the equilibrium strategies.

\section{Monotonicity properties and equilibrium behavior}\label{monotonicity-and-equilibrium}

In this section, we characterize the equilibrium customer behavior using the performance evaluation results for the model under an $(n_e,n_s,q)$-PES that were reported in Section \ref{Performance-evaluation}. We first derive several monotonicity properties of the model associated with the functions $U(n;n_s)$ and $\mathcal{U}(n_e,n_s,q)$ which are crucial for the study of the equilibrium customer behavior.

\begin{prop}\label{monotonicity-with-respect-to-q}
The steady-state number, $N^{-}_q$, of customers in the system at arrival instants during unobservable periods, when the $(n_e,n_s,q)$-PES is followed by the customer population, is stochastically increasing in $q$.
\end{prop}
\begin{proof}
Consider two systems, 1 and 2, with identical operational parameters $\lambda$, $\mu$, $\theta$ and $\zeta$, and identical economic parameters $R$ and $C$, where the customers have adopted the thresholds $n_e$ and $n_s$ given by \eqref{obs-join-threshold} and \eqref{unobs-to-obs-stay-threshold}, respectively. The two systems differ only in the join probability $q$ when customers arrive during an unobservable period. More concretely, we suppose that the customers enter with probability $q^{(i)}$ when they arrive at an unobservable period of system $i$, for $i=1,2$. Suppose that $q^{(1)}<q^{(2)}$ and let $\{(N^{(i)}(t),C^{(i)}(t))\}$ be the CTMC describing system $i$, for $i=1,2$. We construct a coupling of the two processes that represent the states of the two systems as follows: 

The observable and unobservable periods alternate identically in the two systems. The service completions are identical in the two systems and are generated by the same Poisson process $\{M(t)\}$ with rate $\mu$ (when any of the two systems is empty, the Poisson generated events do not have any influence on the corresponding state $0$ of $\{N^{(i)}(t)\}$). The arrivals at system 2 are generated by a Poisson process $\{\Lambda^{(2)}(t)\}$ with rate $\lambda q^{(2)}$. On the other hand, for system 1, we assume that an arrival occurs at an event of the Poisson process $\{\Lambda^{(2)}(t)\}$ with probability $q^{(1)}/q^{(2)}$. This ensures that the arrivals at system 1 constitute a Poisson process $\{\Lambda^{(1)}(t)\}$ with rate $\lambda q^{(1)}$. 

A comparison of  the coupled realizations of the two processes $\{N^{(1)}(t)\}$ and $\{N^{(2)}(t)\}$  shows that both processes move one step to the left, when an event at the service completion Poisson process $\{M(t)\}$ occurs. Moreover, whenever a change happens in the informational process, from the observable to the unobservable mode, whichever of the processes $\{N^{(1)}(t)\}$ and $\{N^{(2)}(t)\}$ are above $n_s$ at the change instant moves to state $n_s$. If only one of them is above $n_s$, then only this process is influenced. Finally, when an event of the Poisson process $\{\Lambda^{(2)}(t)\}$ occurs, the process $\{N^{(2)}(t)\}$ certainly moves one step to the right, while the process $\{N^{(1)}(t)\}$ moves one step to the right with probability $q^{(1)}/q^{(2)}$. Therefore, a moment of reflection shows that if the processes $\{N^{(1)}(t)\}$ and $\{N^{(2)}(t)\}$ start from $n^{(1)}$ and $n^{(2)}$ customers respectively with $n^{(1)}\leq n^{(2)}$, then the sample-path of $\{N^{(1)}(t)\}$ remains `under' the corresponding sample-path  of $\{N^{(2)}(t)\}$ for all $t$. This proves that $\{N^{(1)}(t)\}\leq_{st}\{N^{(2)}(t)\}$. In particular, if we consider the two systems only during their unobservable periods, we conclude that the sample-path of $\{N^{(1)}(t)\}$ remains `under' the corresponding sample-path  of $\{N^{(2)}(t)\}$ so the corresponding steady-state distributions are also stochastically ordered. But during unobservable periods, the arrival processes at both systems are Poisson, so the steady-state distributions of the number of customers in continuous time and at arrival instants coincide (by applying the Poisson-Arrivals-See-Time-Averages (PASTA) result). Therefore, if we denote by $N^-_{q^{(i)}}$ the steady-state distribution of the number of customers in system $i$ at arrival instants when the system is unobservable, $i=1,2$, we conclude that $N^-_{q^{(1)}}\leq_{st} N^-_{q^{(2)}}$.
\end{proof}

\begin{prop}\label{monotonicity-conditional-benefit}
The conditional expected net benefit function $U(n;n_s)$ is strictly decreasing in $n$, for any fixed reneging threshold $n_s$.
\end{prop}
\begin{proof}{Proof} 
The top branch of $U(n;n_s)$ of \eqref{conditional-expected-net-benefit-formula} is obviously strictly decreasing in $n$. The bottom branch is also strictly decreasing in $n$. Indeed, by the definition of $n_s$ (see \eqref{unobs-to-obs-stay-threshold}), we have that $R-f_s-C\frac{n_s}{\mu}\geq r$, so the coefficient $R-r-f_s-\frac{Cn_s}{\mu}+\frac{C}{\theta}$ in the bottom branch of \eqref{conditional-expected-net-benefit-formula} is positive. Moreover, $( \frac{\mu}{\mu+\theta} )^{n+1-n_e}$ is strictly decreasing in $n$. 

It remains to show that $U(n;n_s)$ remains strictly decreasing at its turning point from the top branch to the bottom, i.e., that
\begin{equation*}
R-f_e-f_s-\frac{Cn_s}{\mu}> r-f_e-\frac{C}{\theta}+( R-r-f_s-\frac{Cn_s}{\mu}+\frac{C}{\theta}) \frac{\mu}{\mu+\theta}.
\end{equation*}
After some simplification, this is equivalently written as $R-r-f_s-\frac{Cn_s}{\mu}+\frac{C}{\theta}> 0$ which is valid by the definition of $n_s$. \end{proof}

\begin{prop}\label{monotonicity-unconditional-benefit}
The unconditional expected net benefit function $\mathcal{U}(n_e,n_s,q)$ is strictly decreasing in $q$, for fixed thresholds $n_e$ and $n_s$.
\end{prop}
\begin{proof}
For fixed $n_s$, we have that $U(n;n_s)$ does not depend on $q$, nor on $n_e$, so we can write \eqref{unconditional-benefit-01} as 
\begin{equation*}
\mathcal{U}(n_e,n_s,q)=E[U(N^{-}_{q};n_s)],
\end{equation*}
where $N^-_{q}$ is the random variable that was defined in the statement of Proposition \ref{monotonicity-with-respect-to-q}. Now, $U(n;n_s)$ is strictly decreasing in $n$ because of Proposition \ref{monotonicity-conditional-benefit} and $N^{-}_q$ is stochastically increasing in $q$ by Proposition \ref{monotonicity-with-respect-to-q}. Therefore, $q_1<q_2$ implies that  $\mathcal{U}(n_e,n_s,q_1)=E[U(N^{-}_{q_1};n_s)]>E[U(N^{-}_{q_2};n_s)]=\mathcal{U}(n_e,n_s,q_2)$.
\end{proof}

We are now ready to state and prove the characterization of equilibrium customer strategies.

\begin{thm}\label{thm:equilibrium}
An equilibrium strategy always exists and is unique. It is the $(n_e,n_s,q_e)$-PES with $n_e$, $n_s$ and $q_e$ given respectively from \eqref{obs-join-threshold}, \eqref{unobs-to-obs-stay-threshold} and 
\begin{equation*}
q_e=\left\{
\begin{array}{ll}
0 & \mbox{if } \mathcal{U}(n_e,n_s,0)\leq 0,\\
q_e^* & \mbox{if } \mathcal{U}(n_e,n_s,1)< 0 < \mathcal{U}(n_e,n_s,0),\\
1 & \mbox{if } \mathcal{U}(n_e,n_s,1)\geq 0,
\end{array}
\right.
\end{equation*}
where $q_e^*$ is the root of the equation $\mathcal{U}(n_e,n_s,q)=0$ with respect to $q$ in $(0,1)$ (which exists and is unique when $\mathcal{U}(n_e,n_s,1)< 0 < \mathcal{U}(n_e,n_s,0)$).
\end{thm}
\begin{proof}{Proof}
We have already shown during the discussion in Section \ref{model-strategies}, that an equilibrium strategy is necessarily an $(n_e,n_s,q)$-PES with $n_e$ and $n_s$ given by \eqref{obs-join-threshold} and \eqref{unobs-to-obs-stay-threshold}, respectively. Assume that the population of customers follows the $(n_e,n_s,q)$-PES and consider a tagged customer. We have the following three cases:

\begin{itemize}
\item[Case I:] $q=0$.

Then, the best response of the tagged customer against $(n_e,n_s,0)$ is the same strategy, if and only if, the customer prefers to balk if she finds the system at the unobservable mode upon arrival. Therefore, the $(n_e,n_s,0)$-PES is an equilibrium strategy if and only if $\mathcal{U}(n_e,n_s,0)\leq 0$.

\item[Case II:] $q\in(0,1)$.

Then, the best response of the tagged customer against $(n_e,n_s,q)$ is the same strategy, if and only if, the customer is indifferent between joining and  balking if she finds the system at the unobservable mode upon arrival. Therefore, the $(n_e,n_s,q)$-PES is an equilibrium strategy if and only if $\mathcal{U}(n_e,n_s,q)= 0$. However, because of the strict monotonicity of $\mathcal{U}(n_e,n_s,q)$ with respect to $q$ (by Proposition \ref{monotonicity-unconditional-benefit}), we have that the  equation $\mathcal{U}(n_e,n_s,q)= 0$ has a solution in $(0,1)$ if and only if $\mathcal{U}(n_e,n_s,1)< 0 < \mathcal{U}(n_e,n_s,0)$ and in this case, it is unique.

\item[Case III:] $q=1$.

Then, the best response of the tagged customer against $(n_e,n_s,1)$ is the same strategy, if and only if, the customer prefers to join if she finds the system at the unobservable mode upon arrival. Therefore, the $(n_e,n_s,1)$-PES is an equilibrium strategy if and only if $\mathcal{U}(n_e,n_s,1)\geq 0$.
\end{itemize}\vspace{-0.3cm}
\end{proof}

Theorem~\ref{thm:equilibrium} has far-reaching implications in terms of the economic and operational analysis of the alternating information structure. By establishing existence of a unique equilibrium and by characterizing it in terms of the well behaved customer expected net benefit (cf. Theorem~\ref{thm:unconditional}), it essentially states that the current model remains tractable despite its increased complexity over other benchmark approaches, \cite{HuLiWa2018,BuEcVa2017}. This structure can be utilized to study the model's performance in equilibrium via numerical comparative statics on its operational and economic parameters.

\section{Effects of the alternating observational structure: Numerical experiments}\label{numerical-conclusions}

As previously mentioned, the alternation between observable and unobservable periods represents a number of different situations that arise in practice. These situations determine which parameters can be controlled by a decision maker and hence, the ways in which system performance can be improved. Our objective is to study such effects on strategic customer behavior and to derive managerial insight on optimizing system design.\par
Specifically, we are interested in the behavior of the equilibrium joining probability $q_e$, the equilibrium throughput of the system $\mu_e=\mu(1-p(0,0)-p(0,1))$ (i.e., the number of service completions per time unit) and the equilibrium social welfare
$$S_e=R\mu_e+r a_e-CE[N_{q_e}],$$
where $a_e=\sum_{n=n_s+1}^{\infty}(n-n_s)\theta p(n,0)$ is the mean abandonment (reneging) rate in equilibrium and $E[N_{q_e}]$ is the mean number of customers in the system given by
\begin{align*}
E[N_{q_e}]&=\sum_{n=0}^{n_s}np(n,1)+\sum_{n=0}^{\infty}np(n,0).
\end{align*}

For this analysis, we consider three sets of numerical experiments. The first set studies the effect of the fraction of time that the system is observable in Section~\ref{sub:set1}, whereas the second studies the effect of the duration of the unobservable periods in Section~\ref{sub:set2}. Finally, the third studies the influence of the percentage of the entrance fee which is refundable when a customer reneges (i.e., the fraction $r/f_e$) in Section~\ref{sub:set3}. 


\subsection{Fraction of time that the system is observable}\label{sub:set1}

First, we study the effect on strategic customer behavior of the fraction of time that the system is observable, $\gamma\in [0,1]$, when the mean information cycle of the system -- consisting of an unobservable and an observable period -- is kept fixed. To this end, for a given mean information cycle of the system $B$, we set $\theta=\frac{1}{(1-\gamma) B}$ and $\zeta=\frac{1}{\gamma B}$, so that $1/\theta+1/\zeta=B$, and let $\gamma$ vary in $[0,1]$. Our objective is to study whether there exists an ideal fraction of time, strictly between $0$ and $1$, for which the system should be observable given the duration $B$.\par
We first consider an instance with fixed mean information cycle, $B=0.1$ (high frequency of alternations), and provide the plots of the joining probability $q_e$, the throughput $\mu_e$ and the social welfare $S_e$, when the customers follow the equilibrium strategy, as functions of $\gamma$ in the three panels of Figure~\ref{Fig1}. In each panel, we plot three curves for the arrival rates $\lambda=0.8, 1.1$ and $2.3$, respectively. The rest of the parameters are kept fixed: the service rate is set $\mu=1$ and the economic parameters are $R= 4$, $C=1$, $f_e=f_s=0$ and $r=-30$.
\begin{figure}[!htb]
\centering
\includegraphics[width=\textwidth]{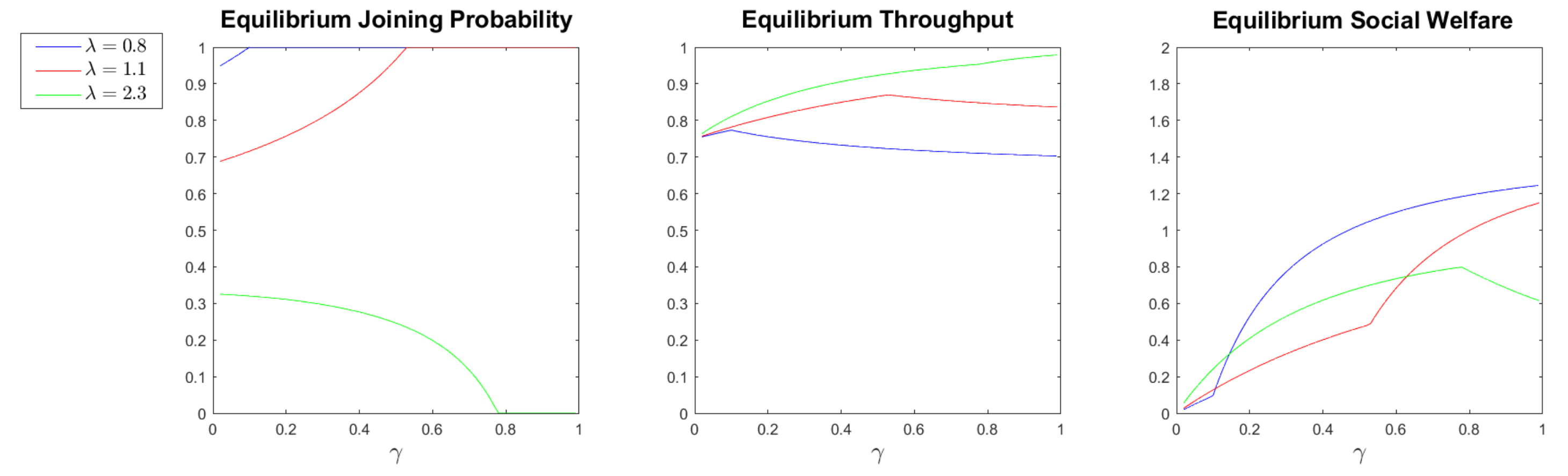}
\caption{Customer's joining probability, throughput and social welfare  with respect to $\gamma$ for $\lambda=0.8, 1.1, 2.3$ and for $B=0.1, \mu=1,\;R=4,\; C=1,\; f_e=f_s=0$ and $r=-30$.}
\label{Fig1}
\end{figure}
The choices $B=0.1$ and $r=-30$ represent that the alternations in the information periods are almost instantaneous and that reneging is prohibited. In this case, $\gamma$ corresponds to the fraction of customers that observe the queue length upon arrival. The rest of the parameters are the same as in Hu, Li and Wang \cite{HuLiWa2018}. The resulting plots of Figure~\ref{Fig1} coincide with their findings, cf. Figures 3, 4 and 5 in \cite{HuLiWa2018}, and confirm that their model can be derived as a limiting case of the alternating information structure. \par

The present model allows a greater degree of control on the fraction of time that the system is observable and hence, it is not surprising that the equilibrium throughput is usually a unimodal function of $\gamma$. This is in agreement with the findings in \cite{HuLiWa2018}, who proved that the equilibrium throughput and the equilibrium social welfare are in general unimodal and not monotonous in the fraction $\gamma$ of informed customers. Similarly, Chen and Frank \cite{ChFr2004} have shown that regarding equilibrium throughput maximization, the observable version of the M/M/1 queue is preferable for high values of $\lambda$, whereas the unobservable version is preferable for low values of $\lambda$. Indeed, when $\lambda$ is high, in the unobservable version no customer enters, whereas some customers do enter in the observable counterpart (those few that find the system in low congestion). The opposite happens when $\lambda$ is low, i.e., all customers enter in the unobservable case, whereas not all customers enter in the observable case.\par
Regulating $\gamma$ and keeping all other parameters fixed has multiple effects on the system. In case of low arrival rates $\lambda$, an increase in $\gamma$ increases the customer entrance probability for the unobservable periods, but it also increases the abandonment probability since the system passes quickly from the unobservable to the observable mode. This double effect is reversed in the case of high arrival rates. Moreover, tuning $\gamma$ changes the composition of the customers' population, which in turn, changes the join-or-balk game among the customers. The population consists of two different subpopulations of customers and the solution of the game becomes more intricate. \par
To understand the effect of $\gamma$ for different frequencies of alternation between observable and unobservable periods, we perform a second experiment with increasing values of the mean information cycle $B=0.1,10,100$. We set the arrival rate at $\lambda=1.1$ (also used by \cite{HuLiWa2018}) and keep the rest of the parameters as in Figure~\ref{Fig1}. The results are shown in Figure~\ref{Fig2}.
\begin{figure}[!htb]
\centering
\includegraphics[width=\textwidth]{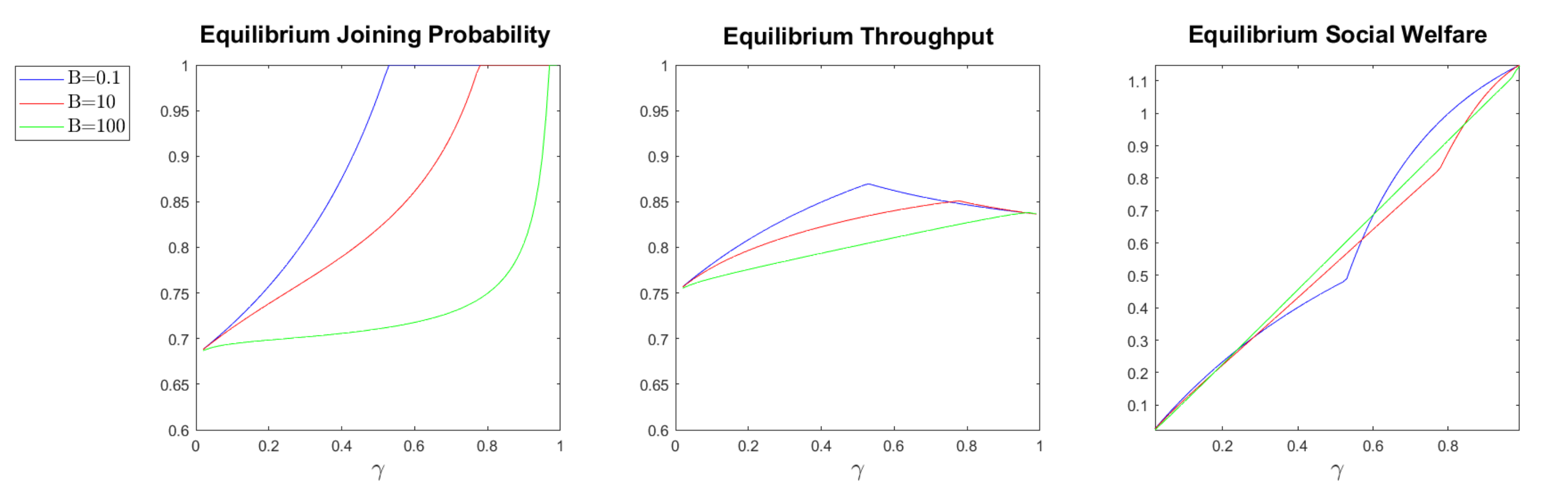}
\caption{Customer's joining probability, throughput and social welfare  with respect to $\gamma$ for $B=0.1,\;10,\;100$ and for $\lambda=1.1,\;\mu=1,\;R=4,\; C=1,\; f_e=f_s=0$ and $r=-30$.}
\label{Fig2}
\end{figure}
The main observation in Figure~\ref{Fig2} is that given sufficient control on the duration of the information cycle, the alternating information structure can lead to increased social welfare in comparison to the information heterogeneity of \cite{HuLiWa2018}. This is achieved for intermediate values of $\gamma$ for which the line for $B=0.1$ (\cite{HuLiWa2018}) lies below the lines of both $B=10$ and $B=100$ (longer information cycles and hence, lower frequency of alternations). To allow comparisons, reneging has been kept prohibited in all these plots by setting $r=-30$. However, allowing customers to renege can further improve system performance.\par
A second feature that is revealed by the plots of Figures~\ref{Fig1} and \ref{Fig2} is the concurrence of the value of $\gamma$ in which the equilibrium joining probability hits $1$, the mode in the equilibrium throughput and the tipping point in social welfare. This remains true for the social welfare if the joining probability hits instead $0$, but not for the equilibrium throughput which remains increasing, cf. plots for $\lambda=2.3$ in Figure~\ref{Fig1} (the terms \emph{increasing} and \emph{decreasing} should be understood in a weak manner, i.e., non-decreasing and non-increasing, respectively). The explanation is based again on the previous discussion: as $\gamma$ increases, it influences the joining and abandonment equilibrium probabilities towards the same direction (both probabilities increase or decrease). But when the joining probability hits its extreme value, $1$ or $0$, as $\gamma$ varies, the effect of further increasing $\gamma$ on the joining probability ceases to exist, whereas the effect on the abandonment probability continues. In short, Figures~\ref{Fig1} and \ref{Fig2} illustrate the following
\begin{itemize}
\item The equilibrium joining probability is increasing in $\gamma$ for low values of $\lambda$ and decreasing in $\gamma$ for high values of $\lambda$. The equilibrium throughput is an increasing or unimodal function of $\gamma$. Its mode coincides with the point at which $q_e$ reaches $1$. The equilibrium social welfare slope changes abruptly when $q_e$ reaches $1$. For certain values of $\gamma$, it is higher for information cycles of higher duration $B$.
\item The equilibrium joining probability is decreasing in $\lambda$, the equilibrium throughput is increasing in $\lambda$, whereas the equilibrium social welfare is non-monotonic in $\lambda$.
\item The equilibrium joining probability is decreasing in $B$, whereas the equilibrium throughput and the equilibrium social welfare are non-monotonic in $B$.
\end{itemize}

\subsection{Duration of unobservable periods}\label{sub:set2}
Next, we turn to the effect of the duration of the unobservable periods, as expressed by varying values of $\theta$, on strategic customer behavior. At the limiting case, in which the duration of the corresponding observable periods is very short (almost instantaneous), i.e., for large values of $\zeta$, the alternating information structure reduces to an \emph{announcement model} with announcement rate $\theta$. This is precisely the setting studied in Burnetas, Economou and Vasiliadis \cite{BuEcVa2017}. In their framework, \cite{BuEcVa2017} show that the equilibrium throughput is a non-monotonic function of the announcement rate $\theta$, when all other parameters are kept fixed. \par
To recover this result, in the first scenario, we let $\theta$ vary in $(0,10)$ for three different arrival rates $\lambda=7,10,40$ and select $\zeta=300$ to model very short observable periods (almost instantaneous announcements). In all cases, the economic parameters $R=5 $, $C=10$, and $f_e=f_s=r=0$ and the operational parameter $\mu=8$ have been kept fixed. The equilibrium perfomance measurers $q_e,\mu_e$ and $S_e$ are plotted as functions of $\theta$ in the three panels of Figure \ref{Fig4}. The three curves in each panel correspond to the different values of $\lambda$.
\begin{figure}[!htb]
\centering
\includegraphics[max width=\linewidth]{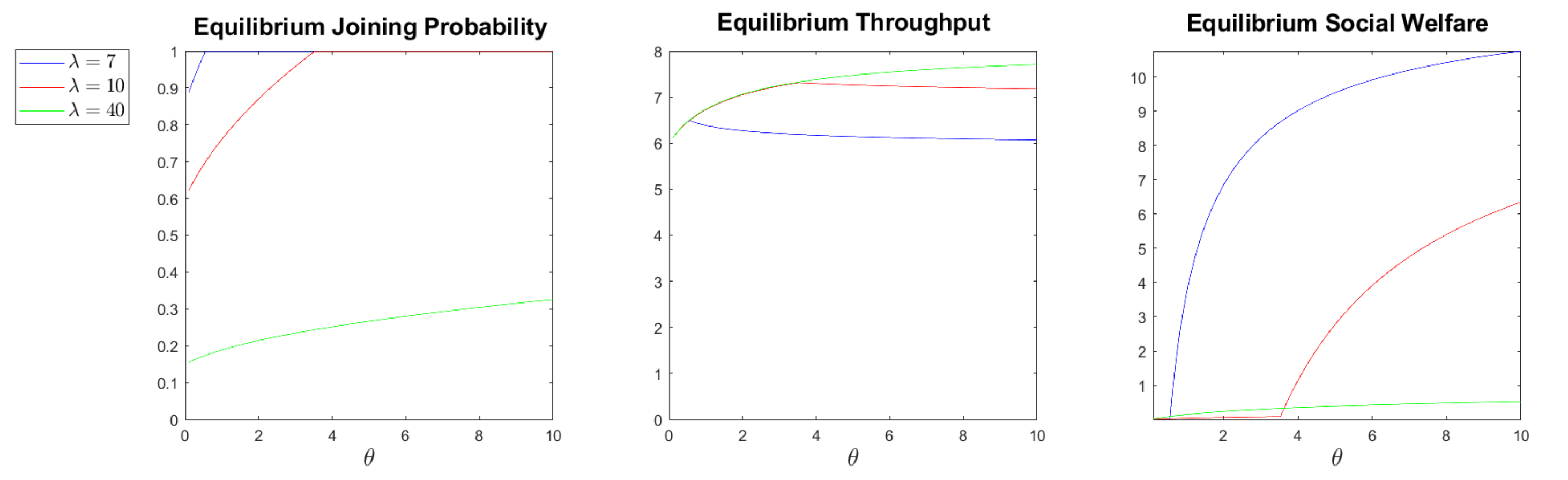}
\caption{Equilibrium joining probability, throughput and social welfare with respect to $\theta$, for $\lambda=7, 10, 40$, when $\zeta=300, \mu=8,\;R=5,\; C=10,\; f_e=f_s=r=0$.}
\label{Fig4}
\end{figure}

The most interesting finding is that the ideal rate $\theta$ for maximizing equilibrium throughput lies strictly between $0$ and $\infty$. This can be seen for the curves that correspond to $\lambda=7$ and $10$ and is also true for the curve with $\lambda=40$ which attains its mode outside the selected range of $\theta$. This happens because increasing $\theta$ has two opposing effects: On the one hand, it increases the equilibrium joining probability, because the uninformed customers become more willing to enter, knowing that they will be informed more quickly. On the other hand, it increases the reneging probability, because an uninformed customer who has joined the system during a high congestion period will abandon the system earlier. The trade-off between the two effects is not clear and this is the reason for the unimodality of the throughput. \par
The flexibility of the current information structure can be utilized to improve over the benchmark model of \cite{BuEcVa2017}. To see this, we consider a second experiment with decreased values of $\zeta=1,10,100$. Lower values of $\zeta$ correspond to longer observable periods. We set the arrival rate at $\lambda=40$, and keep the rest of the parameters as in Figure~\ref{Fig4}. We let $\theta$ to take values in $[0,200]$. The results on the equilibrium performance measures $q_e,\mu_e$ and $S_e$ are shown in the three panels of Figure~\ref{Fig3}.

\begin{figure}[!htb]
\centering
\includegraphics[max width=\linewidth]{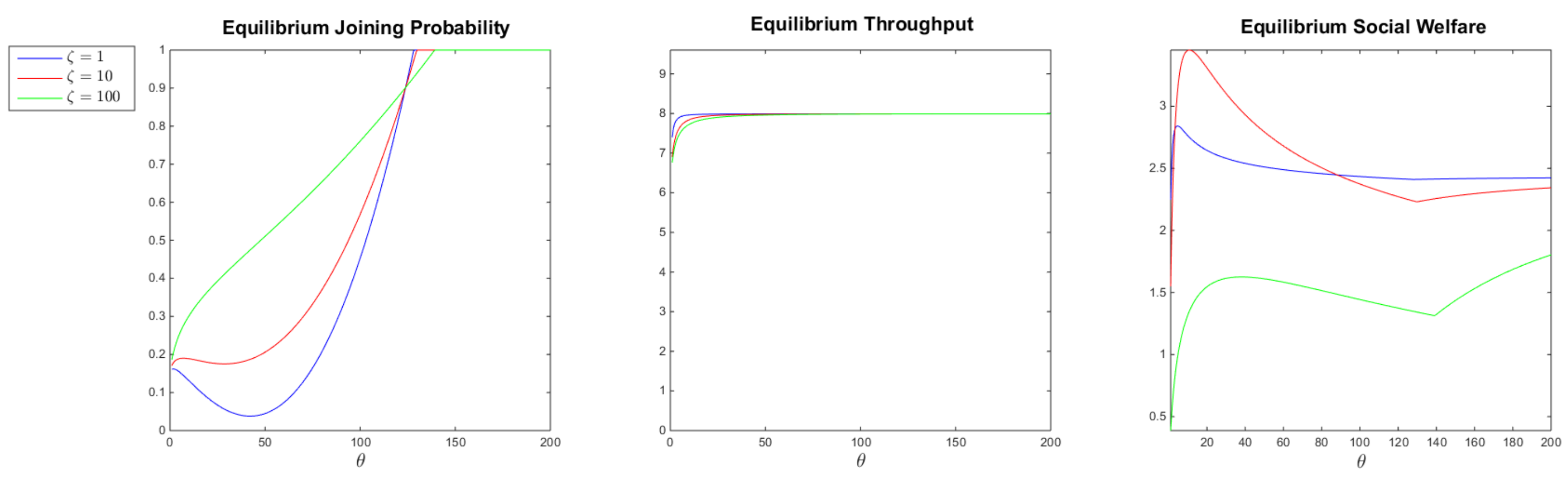}
\caption{Customer's joining probability, throughput and social welfare  with respect to $\theta$ for $\zeta=1, 10, 100$ and for $\lambda=40, \mu=8,\;R=5,\; C=10,\; f_e=f_s=r=0$.}
\label{Fig3}
\end{figure}

The main conclusion that can be drawn from the diversity of the plots is the high dependence of the performance measures, in particular of the equilibrium joining probability and social welfare, on the interplay between the operational model parameters. Due to the their opposing effects, general statements cannot be formulated for broad ranges of parameters' values. For practical situations, this suggests the necessity for a case-by-case analysis to test each set of candidate parameters individually but also indicates the broad applicability of the present model. From an analytic perspective, it underlines the significance to derive the equilibrium of the system and perform in turn comparative statics via the currently employed numerical methods and tools. \par
More concretely, as can be seen from the first panel in Figure~\ref{Fig3}, the equilibrium joining probability, is eventually increasing in $\theta$ and ultimately reaches $1$ for any value of $\zeta$. As $\theta$ increases, the uninformed customers know that they will learn the system state very quickly, so they have a strong incentive to join, independently of the value of $\zeta$. If the observable periods are not long, here $\zeta=100$, then uninformed customers have an incentive to join as $\theta$ increases, since reneging becomes easier. For longer observable periods, here $\zeta=1$, the system becomes occupied by the informed customers which disincentivizes uninformed customers to join and explains the drop in the $\zeta=1$ curve in the first panel for intermediate values of $\theta$.  \par
Finally, the equilibrium throughput quickly reaches its maximal possible value $\mu_e=8$, whereas the social welfare behaves non-monotonically in both $\theta$ and $\zeta$ after an initial steep increase for low values of $\theta$. Indeed, for values of $\theta$ close to $0$, the model essentially corresponds to an unobservable queue. Since the arrival rate, $\lambda=40$, is much higher than the service rate, $\mu=1$, uninformed customers are incentivized to balk, (see also \cite{ChFr2004}). This leaves the system uncongensted for customer arriving at the short observable periods and leads to a sharp increase in equilibrium social welfare. However, as $\theta$ increases further, the effects become mixed. Again, an abrupt change in all curves occurs when the equilibrium joining probability hits $1$. In short, some general statements that can be formulated based on the results in Figures~\ref{Fig4} and \ref{Fig3} are the following
\begin{itemize}
\item The equilibrium joining probability is a non-monotonic function of $\theta$. However, it is eventually increasing in $\theta$ and ultimately reaches $1$. The equilibrium throughput is increasing or unimodal in $\theta$, whereas the equilibrium social welfare is non-monotonic.
\item The equilibrium joining probability is decreasing in $\lambda$, the equilibrium throughput is increasing in $\lambda$ and the social welfare non-monotonic in $\lambda$. 
\item The equilibrium throughput is decreasing in $\zeta$, whereas the equilibrium joining probability and social welfare are non-monotonic in $\zeta$. 
\end{itemize}

\subsection{Fraction of refundable entrance fee}
\label{sub:set3}

In the last set of experiments, we study the effect of the fraction of the entrance fee that is refundable, i.e., of $r/f_e$ for $r\in[0,f_e]$, on the strategic customer behavior. The perfomance measures $q_e, \mu_e$ and $S_e$ have been plotted in the three panels of Figure~\ref{Fig5} as functions of $r/f_e$ in $[0,1]$ for three different service valuations, $R=7, 10$ and $15$. In all cases, the operational parameters have been kept fixed at $\lambda=1.3$, $\mu=1$, $\zeta=10$ and $\theta=1$ and the remaining economic parameters at $C=1$, $f_e=5$ and $f_s=0$. 

\begin{figure}[!htb]
\centering
\includegraphics[width=\linewidth]{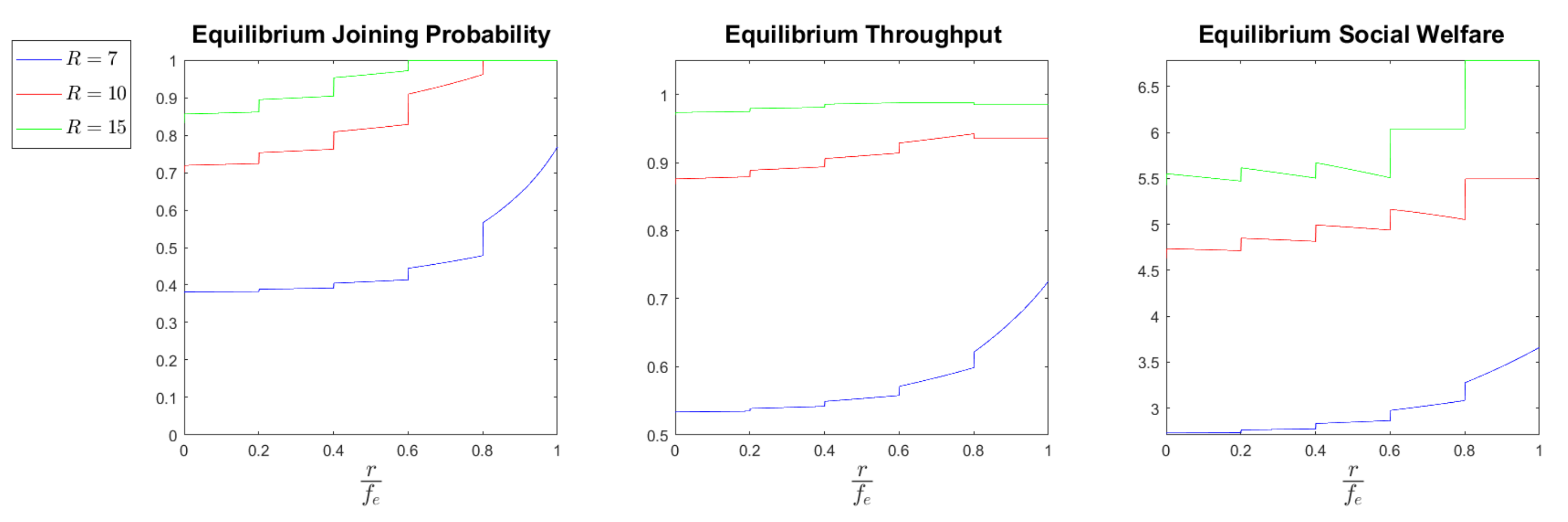}
\caption{Customer's joining probability, throughput and social welfare  with respect to $\frac{r}{f_e}$ for $R=7,\;10,\;15$ and for $\lambda=1.3,\;\mu=1,\;\zeta=10,\;\theta=1,\; C=1,\; f_e=5$ and $f_s=0$.}
\label{Fig5}
\end{figure}

The main finding is the discontinuity of all equilibrium measures which is caused by the discrete changes in the reneging threshold $n_s$ as $r/f_e$ varies. At the points of change, the equilibrium measures undergo abrupt changes or jumps (depicted as vertical lines in the plots). In each interval of continuity, the value of $n_s$ remains the same. Then, an increase in $r/f_e$ makes the uninformed customers more willing to enter and abandon later if they find a high congestion. So, again the joining probability and the abandonment probability both increase and their trade-off is not clear.\par
In a second scenario, we use the same operational and economic parameters $\lambda=1.3,\mu=1,\zeta=10,\theta=1$, and $C=1$, but we now fix $R=7$ and $f_e+f_s=5$. Again, we let $r/f_e$ vary in $[0,1]$. We examine three different scenarios for different values of $f_e=1, 3$ and $5$. Our aim is to model the situation in which the total fee is decomposed in two parts, the entrance and the service fees, and study effect of the percentage of the entrance fee that is refundable. The performance measures are plotted in the three panels of \ref{Fig6}. 

\begin{figure}[!htb]
\centering
\includegraphics[width=\textwidth, trim=0 0 0 0.2cm, clip]{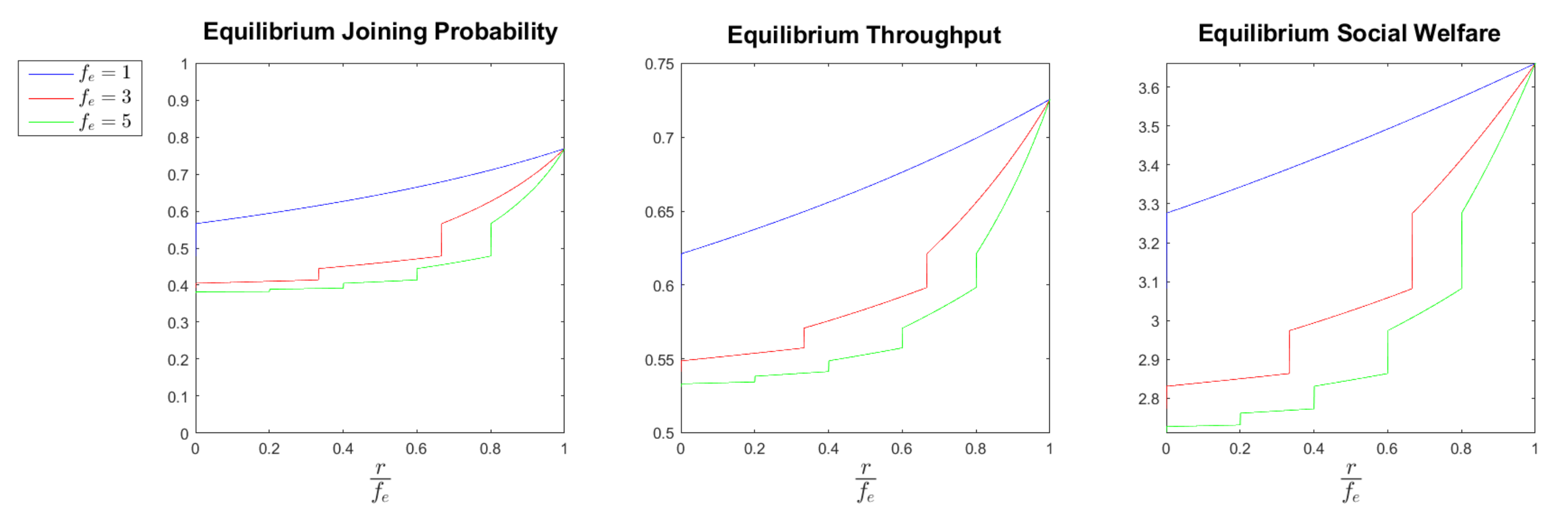}
\caption{Customer's joining probability, throughput and social welfare  with respect to $\frac{r}{f_e}$ for $f_e=1,\;3,\;5$ and for $\lambda=1.3,\;\mu=1,\;\zeta=10,\;\theta=1,\;R=7,\; C=1$ and $f_s=5-f_e$.}
\label{Fig6}
\end{figure}

The rest of the findings of Figures~\ref{Fig5} and \ref{Fig6} are fairly intuitive. All equilibrium performance measures improve as the service valuation $R$ increases or as the service fee $f_s$ increases for a given sum $f_e+f_s$ of entrance and service fees. While the equilibrium joining probability and social welfare also improve as the ratio $r/f_e$ tends to $1$, this may not necessarily be true for the equilibrium throughput, which may decrease, at least marginally, as can be inferred from the $R=15$ and $R=7$ plots in the middle panel of \ref{Fig5}. Yet, in all other cases, the equilibrium increases as well. In short, the findings of Figures~\ref{Fig5} and \ref{Fig6} are the following
\begin{itemize}
\item The equilibrium joining probability is an increasing function of $r/f_e$, the equilibrium throughput is an increasing or unimodal function of $r/f_e$ and the equilibrium social welfare is monotonic in each interval of continuity but in general, it is a non-monotonic function of $r/f_e$. All equilibrium performance measures are discontinuous functions of $r/f_e$.
\item All equilibrium performance measures are increasing in $R$ and in the ratio $f_s/f_e$ for constant $f_e+f_s$.
\end{itemize}
\noindent The results of all experiments are summarized in Table~\ref{tab:results}.

\begin{table}[!htb]
\centering
\renewcommand{\arraystretch}{1.05} 
\begin{tabulary}{0.98\linewidth}{lllp{0.1cm}p{2.2cm}p{2.2cm}l}
\toprule
\b\c Effect&\multicolumn{2}{c}{\b\c Variables}&\c&\multicolumn{3}{c}{\b\c Equilibrium performance measures}\\[0.1cm]
 & axes & plots && $q_e$ & $\mu_e$ & $S_e$\\\midrule
\multirow[t]{4}{*}{\specialcell[t]{Fraction of time \\ that the system\\ is observable}} & \c $\gamma$ &\c &\c&\c\specialcell[t]{\up for low $\lambda$,\\ \down for high $\lambda$} & \c\up or \uni &\c \up or \uni \\ 
&& $\lambda$ && \down &\up & \x\\
&\c& \c$B$ &\c& \c \down & \c \x &\c \x\\
\cmidrule{2-7}
\multirow[t]{3}{*}{\specialcell[t]{Duration of un-\\observable periods}}
 & $\theta$ &&& \x & \up or \uni& \x\\
 &\c&\c $\lambda$ &\c&\c \down &\c \up & \c \x\\
 && $\zeta$ && \x & \down & \x \\\cmidrule{2-7}
\multirow[t]{2}{*}{\specialcell[t]{Fraction of \\refundable fee}}
& $r/f_e$ \c&\c&\c& \c\up & \up or \c\uni &\c \x \\
&& $R$&& \up & \up &\up\\
&\c&\c$f_s/f_e$ &\c&\c \up &\c \up & \c\up\\ 
\bottomrule
\end{tabulary}
\caption{Equilibrium performance measures in the numerical experiments: $q_e$ denotes the equilibrium joining probability, $\mu_e$ the equilibrium throughput and $S_e$ the equilibrium social welfare. Symbol \up stands for non-decreasing, \down for non-increasing, \uni for unimodal and \x for non-monotonic nor unimodal. In column \emph{Variables}, the field \emph{axes} refers to the variables that appear in the horizontal axes of the panels in each figure and the field \emph{plots} to the variables that yield the three different plots in each panel.}
\label{tab:results}
\end{table}

\section{Summary and conclusions}\label{conclusions}

In the present paper, we considered an M/M/1 queue that alternates between exponentially distributed observable and unobservable periods and which bridges the extremal cases of continutously observable and unobservable systems. While this model unifies and generalizes the existing approaches of \cite{HuLiWa2018} and \cite{BuEcVa2017}, it remains analytically tractable since it always has a unique equilibrium that can be characterised via the system parameters. This allows for a comprehensive experimentation on the operational and economic system parameters to gain managerial insight. A main conclusion is that the equilibrium throughput and the corresponding social welfare are typically greater when an ideal level of alternation between observable and unobservable mode is used instead of the system being continuously observable or unobservable. \par
Our results imply that sufficient flexibility to control the information structure of a given queueing system improves its performance both from a managerial and a social perspective. Thus, apart from its practical relevance, the present model  may also provide a benchmark for future studies in this direction. One interesting research problem is to extend the analysis in the case where the unobservable and observable periods are of constant lengths and not exponentially distributed. This seems quite difficult from an analytical point of view, but even a numerical study deserves attention. Another interesting direction for future research concerns the case where the alternation between the observable and unobservable modes of the system is not static (i.e., specified by the exponential rates $\theta$ and $\zeta$) as in the present study, but can be dynamically controlled by the administrator of the system.

\bibliographystyle{acm}

\setstretch{1.4}
\begin{appendices}
\section{Computation of the partial generating functions of the steady-state distribution}\label{generating-functions-appendix}

To obtain the partial generating functions, defined in \eqref{gener-functions-0} and \eqref{gener-functions-1}, of the steady-state probabilities $p(n,i)$, when a given $(n_e,n_s,q)$-PES is used by the population of customers, we use the following approach. First, we multiply the balance equations for $p(n,i)$ in \eqref{bal-1}-\eqref{bal-8} by the appropriate power $z^n$ and add them to obtain equations for the partial generating functions up to a few boundary probabilities to be determined later. We then solve the latter equations using standard algebraic methods.

\subsection{Equations for the partial generating functions}
By multiplying \eqref{bal-2} with $z^n$ and summing for $1\leq n\leq n_e-1$ we obtain
\begin{equation*}
(\lambda+\mu+\zeta)P_{1a}(z)-\mu p(0,1)=\lambda \sum_{n=1}^{n_e-1}p(n-1,1) z^n +\theta \sum_{n=0}^{n_e-1} p(n,0) z^n +\mu \sum_{n=0}^{n_e-1}p(n+1,1) z^n 
\end{equation*}
which reduces after straightforward algebraic manipulations to
\begin{eqnarray}
& &\left[ (\lambda+\mu+\zeta)z -\lambda z^2-\mu \right] P_{1a}(z)-\theta z P_{0a}(z)\nonumber\\
&=&\mu (z-1) p(0,1) -\lambda p(n_e-1,1)z^{n_e+1}+ \mu p(n_e,1)z^{n_e}.\label{eq21-0a,1a-genfunc}
\end{eqnarray}
\noindent 
Similarly, we derive a second equation for the generating functions  $P_{0a}(z)$ and $P_{1a}(z)$, multiplying \eqref{bal-7} by $z^n$ and summing for $1\leq n\leq n_s$, along with \eqref{bal-6}. A bit of algebra yields

\begin{eqnarray}
&&\left[(\lambda q+\mu+\theta)z -\lambda q z^2-\mu \right] P_{0a}(z)-\zeta z P_{1a}(z)\nonumber\\
&=&\mu (z-1) p(0,0) -\lambda q p(n_e-1,0)z^{n_e+1}+ \mu p(n_e,0)z^{n_e}.\label{eq23-0a,1a-genfunc}
\end{eqnarray}
Equations \eqref{eq21-0a,1a-genfunc}, \eqref{eq23-0a,1a-genfunc} can be  written in matrix-form as
\begin{equation}\label{eq28linsyst-a-genfuc}
\left[\begin{array}{cc} -\theta z& (\lambda+\mu+\zeta)z -\lambda z^2-\mu \\ 
(\lambda q+\mu+\theta)z -\lambda q z^2-\mu & -\zeta z 
\end{array}\right]\cdot \left[\begin{array}{l} P_{0a}(z)\\ P_{1a}(z)
\end{array}\right]=\left[\begin{array}{l} N_{0a}(z)\\ N_{1a}(z)
\end{array}\right], 
\end{equation} 
where 
\begin{align}
N_{0a}(z)&=\mu (z-1) p(0,1) -\lambda p(n_e-1,1)z^{n_e+1}+ \mu p(n_e,1)z^{n_e} \label{N0a}\\
N_{1a}(z)&=\mu (z-1) p(0,0) -\lambda q p(n_e-1,0)z^{n_e+1}+ \mu p(n_e,0)z^{n_e}. \label{N1a}
\end{align}
Next, we derive a linear system for the generating functions $P_{0b}(z)$ and $P_{1b}(z)$ following the same procedure. More specifically, multiplying \eqref{bal-4} with $z^{n-n_e}$ for $n_e+1\leq n\leq n_s-1$ and summing them together with \eqref{bal-3} yields
\begin{equation*}
(\mu+\zeta)\sum_{n=n_e}^{n_s-1}p(n,1) z^{n-n_e}=\lambda p(n_e-1,1)+\theta \sum_{n=n_e}^{n_s-1} p(n,0) z^{n-n_e} +\mu \sum_{n=n_e}^{n_s-1}p(n+1,1) z^{n-n_e}.
\end{equation*}
which can be written in a simplified form as
\begin{eqnarray}\label{eq22-0b,1b-genfunc}
\left[ (\mu+\zeta)z -\mu \right] P_{1b}(z)-\theta z P_{0b}(z)=\lambda z p(n_e-1,1) -\mu p(n_e,1)+ \mu p(n_s,1)z^{n_s-n_e}.
\end{eqnarray}
A second equation for $P_{0b}(z)$ and $P_{1b}(z)$, can be derived from \eqref{bal-7}, multiplying with $z^{n-n_e}$ and summing over all $n_e\leq n\leq n_s-1$. It yields  
\begin{eqnarray}
&&\left[ (\lambda q+\mu+\theta)z -\lambda q z^2-\mu \right] P_{0b}(z)-\zeta z P_{1b}(z)\nonumber\\
&=&-\lambda q p(n_s-1,0)z^{n_s-n_e+1}+ \mu p(n_s,0)z^{n_s-n_e}+\lambda q p(n_e-1,0)z- \mu p(n_e,0).\label{eq24-0b,1b-genfunc}
\end{eqnarray}
Again, \eqref{eq22-0b,1b-genfunc}, \eqref{eq24-0b,1b-genfunc}, can be written in matrix-form as
\begin{equation}\label{eq27linsyst-b-genfuc}
\left[\begin{array}{cc} -\theta z& (\mu+\zeta)z -\mu \\ 
(\lambda q+\mu+\theta)z -\lambda q z^2-\mu & -\zeta z 
\end{array}\right]\cdot \left[\begin{array}{l} P_{0b}(z)\\ P_{1b}(z)
\end{array}\right]=\left[\begin{array}{l} N_{0b}(z)\\ N_{1b}(z)
\end{array}\right], 
\end{equation} 

\noindent where 

\begin{align}
N_{0b}(z)&=\mu p(n_s,1)z^{n_s-n_e}+\lambda p(n_e-1,1) z -\mu p(n_e,1) , \label{N0b}\\
N_{1b}(z)&= -\lambda q p(n_s-1,0)z^{n_s-n_e+1}+\mu p(n_s,0)z^{n_s-n_e}+\lambda q p(n_e-1,0)z- \mu p(n_e,0). \label{N1b}
\end{align}
For deriving an equation for $P_{0c}(z)$, we multiply \eqref{bal-7} and \eqref{bal-8} with $z^{n-n_s}$ and sum over all $n\geq n_s$. We have that 
\begin{equation*}
(\lambda q +\mu+\theta) \sum_{n=n_s}^{\infty}p(n,0)z^{n-n_s}= \lambda q \sum_{n=n_s}^{\infty}p(n-1,0)z^{n-n_s}+\zeta p(n_s,1)+\mu \sum_{n=n_s}^{\infty}p(n+1,0)z^{n-n_s},
\end{equation*}  
which reduces easily to 
\begin{equation}\label{eq25-0c-genfunc}
\left[(\lambda q+\mu+\theta) z -\lambda q z^2 - \mu\right] P_{0c}(z)=N_{0c}(z),
\end{equation}
where $ N_{0c}(z)=\lambda q p(n_s-1,0) z+\zeta p(n_s,1) z - \mu p(n_s,0)$. Hence, the balance equations \eqref{bal-1}-\eqref{bal-8} have been transformed into equations \eqref{eq28linsyst-a-genfuc}, \eqref{eq27linsyst-b-genfuc} and \eqref{eq25-0c-genfunc} for the partial generating functions, which can be in turn easily expressed in closed form as rational functions of $z$ via Cramer's rule. It remains to obtain the boundary probabilities that appear in $N_{0a}(z)$, $N_{0b}(z)$, $N_{0c}(z)$, $N_{1a}(z)$ and $N_{1b}(z)$. To this end, we will use the balance equation \eqref{bal-5} and the normalization equation \eqref{norm_cond} that have not been used yet. However, these are only $2$ equations in the $9$ unknown boundary probabilities. The additional required equations will be derived from the roots of the determinants of the linear systems \eqref{eq28linsyst-a-genfuc}, \eqref{eq27linsyst-b-genfuc} and the roots of the coefficient of $P_{0c}(z)$ in \eqref{eq25-0c-genfunc}.

\subsection{Roots of the denominators of the partial generating functions}
Starting with the $P_{0c}(z)$, we define 
\begin{equation}\label{eq_coeff_gen-c}
D_c(z)=(\lambda q+\mu+\theta) z -\lambda q z^2-\mu,
\end{equation}
which is the coefficient of $P_{0c}(z)$ in \eqref{eq25-0c-genfunc}. Since $D_c(0)=-\mu<0$, $D_c(1)=\theta>0$ and $\lim_{z\rightarrow\infty} D_c(z)=-\infty$, it follows from Bolzano's Theorem that there exist real roots $z_{c,1}\in(0,1)$ and $z_{c,2}\in(1,\infty)$ of $D_c(z)$ which are given by
\begin{equation}
z_{c,1},z_{c,2}=\frac{\lambda q+\mu+\theta\mp\sqrt{(\lambda q+\mu+\theta)^2-4\lambda q \mu}}{2\lambda q}.\label{root_gen-c}
\end{equation}
Next, we derive the roots of the determinant $D_b(z)$ of the linear system \eqref{eq27linsyst-b-genfuc}. We have that 

\begin{align*}
D_b(z)&= \det \left[\begin{array}{cc} -\theta z& (\mu+\zeta)z -\mu \\ 
(\lambda q+\mu+\theta)z -\lambda q z^2-\mu & -\zeta z 
\end{array}\right]\\
&=(\lambda q (\mu+\zeta) z^2 - (\lambda q + \mu +\theta +\zeta)\mu z + \mu^2) (z-1).
\end{align*}  
Therefore, $D_b(z)=0$ has three roots, i.e., $z_{b,1}=1$ and
\begin{equation*}
z_{b,2},z_{b,3}=\frac{(\lambda q + \mu +\theta +\zeta)\mu\mp\sqrt{(\lambda q + \mu +\theta +\zeta)^2\mu^2-4\lambda q (\mu+\zeta)\mu^2}}{2\lambda q (\mu+\zeta)}.
\end{equation*}
Similarly, for the determinant $D_a(z)$ of the linear system \eqref{eq28linsyst-a-genfuc} we have 
\begin{align*}
D_a(z)&= \det \left[\begin{array}{cc} -\theta z& (\lambda+\mu+\zeta)z -\lambda z^2-\mu \\ (\lambda q+\mu+\theta)z -\lambda q z^2-\mu & -\zeta z 
\end{array}\right]\\
&=(-\lambda^2 q z^3 + \lambda (\mu+\theta+(\lambda +\mu+\zeta) q ) z^2 - \mu (\lambda q + \mu +\theta +\zeta+\lambda) z + \mu^2)\cdot (z-1).
\end{align*}  
Therefore, $D_a(z)=0$ has four roots, i.e.,  $z_{a,1}=1$ and the three roots of the cubic equation 
\begin{equation*}
-\lambda^2 q z^3 + \lambda (\mu+\theta+(\lambda +\mu+\zeta) q ) z^2 - \mu (\lambda q + \mu +\theta +\zeta+\lambda) z + \mu^2=0,
\end{equation*}
which can be calculated by the general formula for the roots of a cubic equation and are denoted as $z_{a,k}$, for $k=2,3,4$.

\subsection{Computation of the partial generating functions}
In this section, we provide a simple procedure for the computation of the partial generating functions.  Solving \eqref{eq25-0c-genfunc} for  $P_{0c}(z)$ yields
\begin{equation}\label{gen-c_first}
P_{0c}(z)=\frac{(\lambda q p(n_s-1,0) +\zeta p(n_s,1)) z - \mu p(n_s,0)}{\left[(\lambda q+\mu+\theta) z -\lambda q z^2 - \mu\right]}.
\end{equation}    
Since $z_{c,1}$ given by \eqref{root_gen-c} (with the minus sign), is a root of the denominator of $P_{0c}(z)$ inside the closed unit disc, then it should necessarily be a root of its numerator, as $P_{0c}(z)$ is known to converge  in the closed unit disc (as a probability generating function). Hence, the numerator in \eqref{gen-c_first} is a multiple of $z-z_{c,1}$, whereas the denominator can be factored as $-\lambda q (z-z_{c,1})(z-z_{c,2})$. Thus, \eqref{gen-c_first} can be rewritten as 
\begin{equation*}
P_{0c}(z)=\frac{C(z-z_{c,1})}{(z-z_{c,1})(z-z_{c,2})}=\frac{C}{z-z_{c,2}},
\end{equation*}
where $C$ is a constant. But $P_{0c}(0)=p(n_s,0)$, so we conclude that $C=-z_{c,2}\;p(n_s,0)$. Recall, now, that $z_{c,1},z_{c,2}$ are roots of the quadratic equation in \eqref{eq_coeff_gen-c}, thus $z_{c,1}\cdot z_{c,2}=\frac{\mu}{\lambda q}$. Therefore $P_{0c}(z)$ assumes the form

\begin{equation}\label{gen-c_second}
P_{0c}(z)=\frac{p(n_s,0)}{1-\frac{\lambda q z_{c,1}}{\mu} z}=\sum_{n=n_s}^{\infty}p(n_s,0) (\frac{\lambda q z_{c,1}}{\mu})^{n-n_s} z^{n-n_s}. 
\end{equation}
For the derivation of $P_{0b}(z),P_{1b}(z)$ and $P_{0a}(z),P_{1a}(z)$, we apply Cramer's rule to the linear systems \eqref{eq27linsyst-b-genfuc} and \eqref{eq28linsyst-a-genfuc}, respectively. Therefore, for $z\neq z_{1,b},z_{2,b},z_{3,b}$, we have
\begin{equation} \label{eq42linsyst-b-genfuc}
\left[\begin{array}{l} P_{0b}(z)\\ P_{1b}(z)
\end{array}\right]=\left[\begin{array}{l} \frac{-\zeta z N_{0b}(z)-\left[(\mu+\zeta)z-\mu\right]N_{1b}(z)}{D_b(z)}\\ \frac{-\theta z N_{1b}(z)-\left[(\lambda q+\mu+\theta)z-\lambda q z^2-\mu\right]N_{0b}(z)}{D_b(z)} 
\end{array}\right], 
\end{equation} 
and for $z\neq z_{1,a},z_{2,a},z_{3,a}, z_{4,a}$ we have
\begin{equation}\label{eq43linsyst-a-genfuc}
\left[\begin{array}{l} P_{0a}(z)\\ P_{1a}(z)
\end{array}\right]=\left[\begin{array}{l} \frac{-\zeta z N_{0a}(z)-\left[(\lambda +\mu+\zeta)z-\lambda z^2-\mu\right]N_{1a}(z)}{D_a(z)}\\ \frac{-\theta z N_{1a}(z)-\left[(\lambda q+\mu+\theta)z-\lambda q z^2-\mu\right]N_{0a}(z)}{D_a(z)} 
\end{array}\right].
\end{equation} 
Now, we have to compute the boundary probabilities that appear in $N_{0a}(z)$, $N_{0b}(z)$, $N_{1a}(z)$, $N_{1b}(z)$ and $p(n_s,0)$. These are 9 probabilities: $p(0,0)$, $p(0,1)$, $p(n_e-1,0)$, $p(n_e-1,1)$, $p(n_e,0)$, $p(n_e,1)$, $p(n_s-1,0)$, $p(n_s,0)$ and $p(n_s,1)$ (see \eqref{N0a},\eqref{N1a},\eqref{N0b},\eqref{N1b} and \eqref{gen-c_second}).

Rewriting \eqref{bal-5} (that has not been used for the derivation of the equations that govern the partial generating functions) in terms of $P_{0c}(z)$ yields
\begin{equation*}
(\mu+\zeta)p(n_s,1)=\theta P_{0c}(1)=\frac{\theta p(n_s,0)}{1-\frac{\lambda q z_{c,1}}{\mu}}
\end{equation*}
and, we obtain
\begin{equation}\label{eq44_p(ns,1)}
p(n_s,1)=\frac{\theta}{(\mu+\zeta)(1-\frac{\lambda q z_{c,1}}{\mu})}p(n_s,0).
\end{equation}
So, we have expressed $p(n_s,1)$ in terms of $p(n_s,0)$. Next, to obtain the rest 7 unknown probabilities in terms of $p(n_s,0)$, we exploit the fact that the numerator of $P_{0b}(z)$ given in \eqref{eq42linsyst-b-genfuc}, and the numerator of $P_{0a}(z)$ given in  \eqref{eq43linsyst-a-genfuc} should vanish for $z= z_{b,1},z_{b,2},z_{b,3}$ and for $z= z_{a,1},z_{a,2},z_{a,3}, z_{a,4}$, respectively, because these partial generating functions are polynomials and cannot have singularities (poles). Therefore, we obtain the following $7$ equations, one for each root of the corresponding denominator, to obtain the remaining $7$ probabilities $p(0,0)$, $p(0,1)$, $p(n_e-1,0),p(n_e-1,1),p(n_e,0),p(n_e,1)$ and $p(n_s-1,1)$ in terms of $p(n_s,0)$.
These are:
\begin{align}
-\zeta z_{b,i} N_{0b}(z_{b,i})-\left[(\mu+\zeta)z_{b,i}-\mu\right]N_{1b}(z_{b,i})=0, &&& \hbox{ for }i=1,2,3,\label{boundary-prob-eq-1}\\
-\zeta z_{a,i} N_{0a}(z_{a,i})-\left[(\lambda +\mu+\zeta)z_{a,i}-\lambda z_{a,i}^2-\mu\right]N_{1a}(z_{a,i})=0, &&&\hbox{ for }i=1,2,3,4.\label{boundary-prob-eq-2}
\end{align}
Finally, $p(n_s,0)$ is determined using the normalizing equation, and the derivation of the partial generating functions is completed. In practice, one assigns an arbitrary positive value to $p(n_s,0)$ (e.g., $p(n_s,0)=1$), then computes the other boundary probabilities using \eqref{eq44_p(ns,1)} and solving the linear system of \eqref{boundary-prob-eq-1} and \eqref{boundary-prob-eq-2} and finally normalizes the solution so that the total steady-state probability be 1.


\end{appendices}

\end{document}